\documentclass[11pt]{amsart}%
\usepackage{amscd,amsmath,latexsym,amsthm,amsfonts,amssymb,graphicx,color,geometry}
\usepackage{enumerate}
\usepackage[pdftex,plainpages=false,pdfpagelabels,
colorlinks=true,citecolor=blue,linkcolor=blue,urlcolor=black,
filecolor=black,bookmarksopen=true,unicode=true]{hyperref}
\usepackage[norefs,nocites]{refcheck}
\usepackage{amsmath}
\usepackage{amsfonts}
\usepackage{amssymb}
\usepackage[T1]{fontenc}
\usepackage{graphicx}
\usepackage{xcolor}
\usepackage{colortbl}
\usepackage{float}
\usepackage{pgf,tikz,pgfplots}
\usepackage{tkz-base}
\usepackage{tkz-fct}%
\providecommand{\U}[1]{\protect\rule{.1in}{.1in}}
\tikzset{>=Triangle}
\newtheorem{theorem}{Theorem}[section]

\newtheorem{corollary}[theorem]{Corollary}

\newtheorem{remark}[theorem]{Remark}

\newtheorem{lemma}[theorem]{Lemma}

\numberwithin{equation}{section}
\pgfplotsset{compat=1.17}
\begin{document}
\title[Constants of the Kahane--Salem--Zygmund inequality]{Constants of the Kahane--Salem--Zygmund inequality asymptotically bounded by
$1$}
\author[D. Pellegrino]{Daniel Pellegrino}
\address{Departamento de Matem\'{a}tica \\
Universidade Federal da Para\'{\i}ba \\
58.051-900 - Jo\~{a}o Pessoa, Brazil.}
\email{daniel.pellegrino@academico.ufpb.br}
\author[A. Raposo Jr.]{Anselmo Raposo Jr.}
\address{Departamento de Matem\'{a}tica \\
Universidade Federal do Maranh\~{a}o \\
65085-580 - S\~{a}o Lu\'{\i}s, Brazil.}
\email{anselmo.junior@ufma.br}
\thanks{D. Pellegrino is supported by CNPq and Grant 2019/0014 Paraiba State Research
Foundation (FAPESQ) }

\begin{abstract}
The Kahane--Salem--Zygmund inequality for multilinear forms in $\ell_{\infty}$
spaces claims that, for all positive integers $m,n_{1},...,n_{m}$, there
exists an $m$-linear form $A\colon\ell_{\infty}^{n_{1}}\times\cdots\times
\ell_{\infty}^{n_{m}}\longrightarrow\mathbb{K}$ \ ($\mathbb{K}=\mathbb{R}$ or
$\mathbb{C}$) of the type
\[
A(z^{(1)},...,z^{(m)})=\sum_{j_{1}=1}^{n_{1}}\cdots\sum_{j_{m}=1}^{n_{m}}\pm
z_{j_{1}}^{\left(  1\right)  }\cdots z_{j_{m}}^{\left(  m\right)  }\text{,}%
\]
satisfying
\[
\Vert A\Vert\leq C_{m}\max\left\{  n_{1}^{1/2},\ldots,n_{m}^{1/2}\right\}
{\textstyle\prod\limits_{j=1}^{m}}n_{j}^{1/2}\text{,}%
\]
for
\[
C_{m}\leq\kappa\sqrt{m\log m}\sqrt{m!}%
\]
and a certain $\kappa>0.$ Our main result shows that given any $\epsilon>0$
and any positive integer $m,$ there exists a positive integer $N$ such that
\[
C_{m}<1+\epsilon\text{,}%
\]
when we consider $n_{1},...,n_{m}>N$. In addition, while the original proof of
the Kahane--Salem--Zygmund relies in highly non-deterministic arguments, our
approach is constructive. We also provide the same asymptotic bound (which is
shown to be optimal in some cases) for the constant of a related
non-deterministic inequality proved by G. Bennett in 1977. Applications to
Berlekamp's switching game are given.

\end{abstract}
\maketitle
\tableofcontents

\section{Introduction}

All along this paper we shall denote $\mathbb{K}=\mathbb{R}$ or $\mathbb{C}$
and we shall represent $\mathbb{K}^{n}$ endowed with the $\ell_{p}$-norm by
$\ell_{p}^{n}$; the set of all positive integers shall be denoted by
$\mathbb{N}$. The investigation of bilinear forms and, more generally,
multilinear forms $A\colon\ell_{p_{1}}^{n}\times\cdots\times\ell_{p_{m}}%
^{n}\longrightarrow\mathbb{K}$ with coefficients $\pm1$ has important
applications in several fields of mathematics and has been explored by several
authors in different contexts since the beginning of the mid last century.

The ideas that guide this topic rest on the search of unimodular multilinear
forms with relatively small norms, i.e., on the existence of multilinear forms
with coefficients $1$ and $-1$ (sometimes the coefficients are allowed to be
complex numbers with modulus $1$) and smallest possible norms. In 1931,
Bohnenblust and Hille \cite{bh} constructed an $m$-linear form $A_{n}%
\colon\ell_{\infty}^{n}\times\cdots\times\ell_{\infty}^{n}\longrightarrow
\mathbb{C}$ with complex coefficients with modulus $1$ satisfying%
\[
\left\Vert A_{n}\right\Vert \leq n^{(m+1)/2}%
\]
and they have shown that the exponent $\left(  m+1\right)  /2$ is optimal,
i.e., it cannot be replaced by a smaller one (even if we multiply the
right-hand-side of the inequality by a positive constant). However, in many
instances, deterministic arguments are not available and probabilistic methods
come into play. In the 1970's and 1980's several authors have followed this
vein (see \cite{be2,carl,kah,man,var}) and nowadays these probabilistic
approaches are explored in different lines of research (see
\cite{alon,bayart,boasfootball,boas,bk,TBJS,dgm,defantlivro,defmas,mastylo,PRS}
and the references therein). In general, the probabilistic arguments are
enough to provide optimal exponents but the constants involved are not precise.

These classes of inequalities providing unimodular multilinear forms (and
polynomials) with small norms by means of probabilistic methods are usually
called Kahane--Salem--Zygmund inequalities (KSZ inequalities for short).
Recently, important far-reaching generalizations of these inequalities have
been developed by Masty\l o--Szwedek (\cite{mastylo}) and Defant--Masty\l o
(\cite{defmas}).

In this paper we shall be interested in the multilinear versions of the
KSZ\ inequality (see \cite[Lemma 6.1]{jfa} and \cite{alb,PRS}, which are
obtained following the approach of Boas in \cite[Theorem 4]{boas}; see also
\cite[Theorems 1.1 and 1.2]{man}). For the case of bilinear forms, a KSZ-type
inequality was independently proved by Bennett \cite{be1} in 1977 (see also
\cite{be2}), in a quite complete fashion. Applications of the KSZ inequality
can be found, for instance, in \cite{israel,bk,carl,vel,var}. The KSZ
inequality for multilinear forms asserts that for all positive integers $m,n$
and $p_{1},\ldots,p_{m}\in\lbrack1,\infty]$, there exists an $m$-linear form
$A_{n}\colon\ell_{p_{1}}^{n}\times\cdots\times\ell_{p_{m}}^{n}\longrightarrow
\mathbb{K}$ of the type
\[
A_{n}(z^{(1)},\ldots,z^{(m)})=%
%TCIMACRO{\tsum \limits_{i_{1}=1}^{n}}%
%BeginExpansion
{\textstyle\sum\limits_{i_{1}=1}^{n}}
%EndExpansion
\cdots%
%TCIMACRO{\tsum \limits_{i_{m}=1}^{n}}%
%BeginExpansion
{\textstyle\sum\limits_{i_{m}=1}^{n}}
%EndExpansion
\pm z_{i_{1}}^{(1)}\cdots z_{i_{m}}^{(m)}\text{,}%
\]
such that
\begin{equation}
\Vert A_{n}\Vert\leq C_{m}n^{\frac{1}{\min\{\max\{2,p_{1}^{\ast}\},\ldots
,\max\{2,p_{m}^{\ast}\}\}}+\sum_{k=1}^{m}\max\left\{  \frac{1}{2}-\frac
{1}{p_{k}},0\right\}  }\text{,} \label{I1}%
\end{equation}
with%
\begin{equation}
C_{m}=\sqrt{32m\log\left(  6m\right)  }\sqrt{m!}\text{.} \label{I2}%
\end{equation}
Above and henceforth, as usual, we consider $1/\infty=0$, the conjugate of $p$
is denoted by $p^{\ast}$, i.e., $p^{\ast}=p/\left(  p-1\right)  $.

For the case of bilinear forms Bennett's approach \cite[Proposition 3.2]{be1}
is more general, allowing different dimensions at the domain of the bilinear
forms. More precisely, Bennett's inequality affirms that there is a constant
$C\geq1$ such that for all $p_{1},p_{2}\in\lbrack1,\infty]$ and all positive
integers $n_{1},n_{2}$, there exists a bilinear form $A_{n_{1},n_{2}}%
\colon\ell_{p_{1}}^{n_{1}}\times\ell_{p_{2}}^{n_{2}}\longrightarrow\mathbb{R}$
with coefficients $\pm1$ satisfying%
\begin{equation}
\left\Vert A_{n_{1},n_{2}}\right\Vert \leq C\max\left\{  n_{2}^{1/p_{2}^{\ast
}}n_{1}^{\max\left\{  \frac{1}{2}-\frac{1}{p_{1}},0\right\}  },n_{1}%
^{1/p_{1}^{\ast}}n_{2}^{\max\left\{  \frac{1}{2}-\frac{1}{p_{2}},0\right\}
}\right\}  \text{.} \label{I3}%
\end{equation}
The KSZ inequality (following Bennett's style) was recently extended to
$m$-linear forms (\cite{alb,PRS}) as follows (see also \cite{defmas} for a
recent far-reaching generalization): let $m,n_{1},\ldots,n_{m}$ be positive
integers and $p_{1},\ldots,p_{m}\in\left[  1,\infty\right]  $. There exist a
constant $C_{m}$ and an $m$-linear form $A_{n_{1},...,n_{m}}\colon\ell_{p_{1}%
}^{n_{1}}\times\cdots\times\ell_{p_{m}}^{n_{m}}\longrightarrow\mathbb{K}$ of
the type
\[
A_{n_{1},...,n_{m}}\left(  z^{(1)},\ldots,z^{(m)}\right)  =%
%TCIMACRO{\tsum \limits_{i_{1}=1}^{n_{1}}}%
%BeginExpansion
{\textstyle\sum\limits_{i_{1}=1}^{n_{1}}}
%EndExpansion
\cdots%
%TCIMACRO{\tsum \limits_{i_{m}=1}^{n_{m}}}%
%BeginExpansion
{\textstyle\sum\limits_{i_{m}=1}^{n_{m}}}
%EndExpansion
\pm z_{i_{1}}^{(1)}\cdots z_{i_{m}}^{(m)}\text{,}%
\]
such that%
\begin{equation}
\Vert A_{n_{1},...,n_{m}}\Vert\leq C_{m}\left(
%TCIMACRO{\tsum \limits_{k=1}^{m}}%
%BeginExpansion
{\textstyle\sum\limits_{k=1}^{m}}
%EndExpansion
n_{k}\right)  ^{\frac{1}{\min\{\max\{2,p_{1}^{\ast}\},\ldots,\max
\{2,p_{m}^{\ast}\}\}}}%
%TCIMACRO{\tprod \limits_{k=1}^{m}}%
%BeginExpansion
{\textstyle\prod\limits_{k=1}^{m}}
%EndExpansion
n_{k}^{\max\left\{  \frac{1}{2}-\frac{1}{p_{k}},0\right\}  }\text{.}
\label{I4}%
\end{equation}

It is obvious that (\ref{I4}) is equivalent to write
\begin{equation}
\Vert A_{n_{1},...,n_{m}}\Vert\leq D_{m}\max_{k=1,\ldots,m}\left\{
n_{k}^{\frac{1}{\min\{\max\{2,p_{1}^{\ast}\},\ldots,\max\{2,p_{m}^{\ast}\}\}}%
}\right\}
%TCIMACRO{\tprod \limits_{k=1}^{m}}%
%BeginExpansion
{\textstyle\prod\limits_{k=1}^{m}}
%EndExpansion
n_{k}^{\max\left\{  \frac{1}{2}-\frac{1}{p_{k}},0\right\}  } \label{I5}%
\end{equation}
for some constant $D_{m}\geq C_{m}$. A straightforward computation shows that
when $m=2$ and $p_{1},p_{2}\in\lbrack2,\infty]$, the inequality (\ref{I5})
recovers (\ref{I3}). The proof of the above inequalities follows the lines of
the proof of (\ref{I1}) and the constants $C_{m},D_{m}$ are essentially the
constant that appears in (\ref{I2}).

The exponent of $n$ in (\ref{I1}) is optimal (see \cite{PRS}; in fact this
result was previously proved by Mantero--Tonge \cite[Theorem 1.2]{man}) and,
when $p_{1},\ldots,p_{m}\in\lbrack2,\infty]$, all the exponents of
$n_{1},\ldots,n_{m}$ in all the expressions above are also sharp (see
\cite{alb}). On the other hand, the exact values of the constants
$C,C_{m},D_{m}$ are unknown; the approaches of Bennett, Boas and
Mantero--Tonge (\cite{be1,boas,man}) are non-constructive and rely on highly
non deterministic methods, giving no precise hint of the exact value of the
constants involved.

In this paper, among other results, we show that, provided that $p_{1}%
=\cdots=p_{m}=\infty$, we have (\ref{I1}), (\ref{I3}) and (\ref{I5}) with
constants \textquotedblleft asymptotically dominated\textquotedblright\ by 1.
More precisely, our first main result provides the aforementioned asymptotic
bounds for (\ref{I3}) and (\ref{I5}) when $p_{1}=\cdots=p_{m}=\infty$ and
reads as follows:

\begin{theorem}
\label{090}Let a positive integer $m$ and $\epsilon>0$ be given. There exists
a positive integer $N$ such that, for all $n_{1},...,n_{m}>N$, there exists an
$m$-linear form $A_{n_{1},\ldots,n_{m}}\colon\ell_{\infty}^{n_{1}}\times
\cdots\times\ell_{\infty}^{n_{m}}\longrightarrow\mathbb{K}$ of the type
\[
A_{n_{1},\ldots,n_{m}}(z^{(1)},...,z^{(m)})=%
%TCIMACRO{\tsum \limits_{i_{1}=1}^{n_{1}}}%
%BeginExpansion
{\textstyle\sum\limits_{i_{1}=1}^{n_{1}}}
%EndExpansion
\cdots%
%TCIMACRO{\tsum \limits_{i_{m}=1}^{n_{m}}}%
%BeginExpansion
{\textstyle\sum\limits_{i_{m}=1}^{n_{m}}}
%EndExpansion
\pm z_{i_{1}}^{(1)}\cdots z_{i_{m}}^{(m)}\text{,}%
\]
such that%
\[
\Vert A_{n_{1},\ldots,n_{m}}\Vert\leq\left(  1+\epsilon\right)  \max\left\{
n_{1}^{1/2},\ldots,n_{m}^{1/2}\right\}
%TCIMACRO{\tprod \limits_{j=1}^{m}}%
%BeginExpansion
{\textstyle\prod\limits_{j=1}^{m}}
%EndExpansion
n_{j}^{1/2}.
\]

\end{theorem}

As an immediate consequence we have the same asymptotic bound for (\ref{I1}):

\begin{corollary}
Let a positive integer $m$ and $\epsilon>0$ be given. There exists a positive
integer $N$ such that, for all $n>N$, there exists an $m$-linear form
$A_{n}\colon\ell_{\infty}^{n}\times\cdots\times\ell_{\infty}^{n}%
\longrightarrow\mathbb{K}$ of the type
\[
A_{n}(z^{(1)},...,z^{(m)})=%
%TCIMACRO{\tsum \limits_{i_{1}=1}^{n}}%
%BeginExpansion
{\textstyle\sum\limits_{i_{1}=1}^{n}}
%EndExpansion
\cdots%
%TCIMACRO{\tsum \limits_{i_{m}=1}^{n}}%
%BeginExpansion
{\textstyle\sum\limits_{i_{m}=1}^{n}}
%EndExpansion
\pm z_{i_{1}}^{(1)}\cdots z_{i_{m}}^{(m)}\text{,}%
\]
such that
\[
\Vert A_{n}\Vert\leq\left(  1+\epsilon\right)  n^{\frac{m+1}{2}}.
\]

\end{corollary}

For the particular case of bilinear forms we prove a more general result
(which, in particular, shows that the constants of Bennett's inequality are
\textquotedblleft uniformly\textquotedblright\ asymptotically bounded by $1$).
For instance, among other results encompassing different dimensions
$n_{1},n_{2}$ (see Figure 2), we show that if $p_{1},p_{2}\in\lbrack1,\infty]$
and $\epsilon>0$, there is a positive integer $N$ (depending just on
$\epsilon$) such that, whenever $n>N$, there is a bilinear form $A\colon
\ell_{p_{1}}^{n}\times\ell_{p_{2}}^{n}\longrightarrow\mathbb{K}$ of the type%
\[
A(z^{(1)},z^{(2)})=%
%TCIMACRO{\tsum \limits_{i=1}^{n}}%
%BeginExpansion
{\textstyle\sum\limits_{i=1}^{n}}
%EndExpansion%
%TCIMACRO{\tsum \limits_{j=1}^{n}}%
%BeginExpansion
{\textstyle\sum\limits_{j=1}^{n}}
%EndExpansion
\pm z_{i}^{(1)}z_{j}^{(2)}%
\]
such that%
\[
\Vert A\Vert\leq\left(  1+\epsilon\right)  \max\left\{  n^{1/p_{2}^{\ast}%
}n^{\max\left\{  \frac{1}{2}-\frac{1}{p_{1}},0\right\}  },n^{1/p_{1}^{\ast}%
}n^{\max\left\{  \frac{1}{2}-\frac{1}{p_{2}},0\right\}  }\right\}
\]
and the constant $1$ is optimal whenever $p_{1},p_{2}\in\lbrack1,2]$.

The paper is organized as follows. The proof of Theorem \ref{090} is presented
in Section 2, and the case of bilinear forms is considered in Section 3. In
Section 4 we present some remarks, including an improvement of the constants
of (\ref{I5}) and in the final section we apply our results to the
Gale--Berlekamp switching game.

\section{The proof of Theorem \ref{090}}

We recall that a Hadamard matrix of order $n$ is a square matrix
$\mathbf{H=}\left[  h_{ij}\right]  _{n\times n}$, with $h_{ij}\in\left\{
-1,1\right\}  $ for all $i,j$, such that
\[
\mathbf{HH}^{\top}=n\mathbf{I}_{n}\text{,}%
\]
where $\mathbf{I}_{n}$ is the identity matrix of order $n$ and $\mathbf{H}%
^{\top}$ is the transpose of $\mathbf{H}$. Thus, if $u_{1},\ldots,u_{n}$ are
the rows of $\mathbf{H}$, then the inner product of the rows is
\[
\left\langle u_{i},u_{j}\right\rangle =n\delta_{ij}%
\]
for all $i,j$, where $\delta_{ij}$ denotes the Kronecker delta. In particular,
the rows of a Hadamard matrix are pairwise orthogonal. For more details we
refer to the book \cite{KJH}. We begin with the following lemma that adapts a
technique borrowed from \cite{bh} (see also \cite{PRS}):

\begin{lemma}
\label{LemaM1}Let $n_{1}\leq n_{2}\leq\cdots\leq n_{m}$ be positive integers
such that for each $k=2,\ldots,m$, there exists a Hadamard matrix
$\mathbf{H}_{n_{k}}=\left[  h_{ij}^{\left(  k\right)  }\right]  _{n_{k}\times
n_{k}}$. Then, the $m$-linear form
\[
A_{0}\colon\ell_{\infty}^{n_{1}}\times\cdots\times\ell_{\infty}^{n_{m}%
}\longrightarrow\mathbb{C}%
\]
defined by%
\[
A_{0}\left(  x^{(1)},\dots,x^{(m)}\right)  =%
%TCIMACRO{\tsum \limits_{i_{1}=1}^{n_{1}}}%
%BeginExpansion
{\textstyle\sum\limits_{i_{1}=1}^{n_{1}}}
%EndExpansion
\cdots%
%TCIMACRO{\tsum \limits_{i_{m}=1}^{n_{m}}}%
%BeginExpansion
{\textstyle\sum\limits_{i_{m}=1}^{n_{m}}}
%EndExpansion
h_{i_{1}i_{2}}^{(2)}h_{i_{2}i_{3}}^{(3)}\cdots h_{i_{m-1}i_{m}}^{(m)}x_{i_{1}%
}^{(1)}\cdots x_{i_{m}}^{(m)}%
\]
has norm%
\[
\left\Vert A_{0}\right\Vert \leq n_{m}^{1/2}%
%TCIMACRO{\tprod \limits_{j=1}^{m}}%
%BeginExpansion
{\textstyle\prod\limits_{j=1}^{m}}
%EndExpansion
n_{j}^{1/2}\text{.}%
\]

\end{lemma}

\begin{proof}
For $k=2,\ldots,m$, let $u_{i}^{\left(  k\right)  }$, $i=1,\ldots,n_{k}$, be
the rows of $\mathbf{H}_{n_{k}}$; hence
\begin{equation}
\left\langle u_{i}^{\left(  k\right)  },u_{j}^{\left(  k\right)
}\right\rangle =n_{k}\delta_{ij}\text{.} \label{M1}%
\end{equation}
Let us consider the square matrices of order $n_{m}$ defined by%
\[
\left[  h_{ij}^{\left(  k\right)  }\right]  _{n_{m}\times n_{m}}:=\left[
\begin{array}
[c]{cc}%
\mathbf{H}_{n_{k}} & \mathbf{0}_{n_{k}\times\left(  n_{m}-n_{k}\right)  }\\
\mathbf{0}_{\left(  n_{m}-n_{k}\right)  \times n_{k}} & \mathbf{0}_{\left(
n_{m}-n_{k}\right)  \times\left(  n_{m}-n_{k}\right)  }%
\end{array}
\right]
\]
for each $k=2,\ldots,m$. Define%
\[
A_{0}\colon\ell_{\infty}^{n_{1}}\times\cdots\times\ell_{\infty}^{n_{m}%
}\longrightarrow\mathbb{C}%
\]
by
\[
A_{0}\left(  x^{(1)},\dots,x^{(m)}\right)  =%
%TCIMACRO{\tsum \limits_{i_{1},\dots,i_{m}=1}^{n_{1},\ldots,n_{m}}}%
%BeginExpansion
{\textstyle\sum\limits_{i_{1},\dots,i_{m}=1}^{n_{1},\ldots,n_{m}}}
%EndExpansion
h_{i_{1}i_{2}}^{(2)}h_{i_{2}i_{3}}^{(3)}\cdots h_{i_{m-1}i_{m}}^{(m)}x_{i_{1}%
}^{(1)}\cdots x_{i_{m}}^{(m)}\text{,}%
\]
and note that
\[
h_{i_{1}i_{2}}^{(2)}h_{i_{2}i_{3}}^{(3)}\cdots h_{i_{m-1}i_{m}}^{(m)}%
\in\left\{  -1,1\right\}
\]
whenever $i_{k}\in\{1,\ldots,n_{k}\}$ for all $k=1,\ldots,m$. For each
$k=1,\ldots,m$, given $x^{(k)}$ in the closed unit ball $B_{\ell_{\infty
}^{n_{k}}}$, let us denote
\[
y^{(k)}=\left(  x_{1}^{(k)},\ldots,x_{n_{k}}^{(k)},0,\ldots,0\right)  \in
B_{\ell_{\infty}^{n_{m}}}\text{.}%
\]
Then, by the H\"{o}lder inequality,%
\begin{align*}
&  \left\vert A_{0}\left(  x^{(1)},\dots,x^{(m)}\right)  \right\vert
=\left\vert
%TCIMACRO{\tsum \limits_{i_{1},\dots,i_{m}=1}^{n_{m}}}%
%BeginExpansion
{\textstyle\sum\limits_{i_{1},\dots,i_{m}=1}^{n_{m}}}
%EndExpansion
h_{i_{1}i_{2}}^{(2)}h_{i_{2}i_{3}}^{(3)}\cdots h_{i_{m-1}i_{m}}^{(m)}y_{i_{1}%
}^{(1)}\cdots y_{i_{m}}^{(m)}\right\vert \\
&  \leq%
%TCIMACRO{\tsum \limits_{i_{m}=1}^{n_{m}}}%
%BeginExpansion
{\textstyle\sum\limits_{i_{m}=1}^{n_{m}}}
%EndExpansion
\left\vert
%TCIMACRO{\tsum \limits_{i_{1},\dots,i_{m-1}=1}^{n_{m}}}%
%BeginExpansion
{\textstyle\sum\limits_{i_{1},\dots,i_{m-1}=1}^{n_{m}}}
%EndExpansion
h_{i_{1}i_{2}}^{(2)}h_{i_{2}i_{3}}^{(3)}\cdots h_{i_{m-1}i_{m}}^{(m)}y_{i_{1}%
}^{(1)}\cdots y_{i_{m-1}}^{(m-1)}\right\vert \left\vert y_{i_{m}}^{\left(
m\right)  }\right\vert \\
&  \leq\left(
%TCIMACRO{\tsum \limits_{i_{m}=1}^{n_{m}}}%
%BeginExpansion
{\textstyle\sum\limits_{i_{m}=1}^{n_{m}}}
%EndExpansion
|y_{i_{m}}^{(m)}|^{2}\right)  ^{1/2}\cdot\left(
%TCIMACRO{\tsum \limits_{i_{m}=1}^{n_{m}}}%
%BeginExpansion
{\textstyle\sum\limits_{i_{m}=1}^{n_{m}}}
%EndExpansion
\left\vert
%TCIMACRO{\tsum \limits_{i_{1},\dots,i_{m-1}=1}^{n_{m}}}%
%BeginExpansion
{\textstyle\sum\limits_{i_{1},\dots,i_{m-1}=1}^{n_{m}}}
%EndExpansion
h_{i_{1}i_{2}}^{(2)}h_{i_{2}i_{3}}^{(3)}\cdots h_{i_{m-1}i_{m}}^{(m)}y_{i_{1}%
}^{(1)}\cdots y_{i_{m-1}}^{(m-1)}\right\vert ^{2}\right)  ^{1/2}\\
&  \leq n_{m}^{1/2}\left(
%TCIMACRO{\tsum \limits_{i_{m}=1}^{n_{m}}}%
%BeginExpansion
{\textstyle\sum\limits_{i_{m}=1}^{n_{m}}}
%EndExpansion%
%TCIMACRO{\tsum \limits_{\substack{i_{1},\dots,i_{m-1}=1\\j_{1},\dots
%,j_{m-1}=1}}^{n_{m}}}%
%BeginExpansion
{\textstyle\sum\limits_{\substack{i_{1},\dots,i_{m-1}=1\\j_{1},\dots
,j_{m-1}=1}}^{n_{m}}}
%EndExpansion
h_{i_{1}i_{2}}^{(2)}h_{j_{1}j_{2}}^{(2)}\cdots h_{i_{m-2}i_{m-1}}%
^{(m-1)}h_{j_{m-2}j_{m-1}}^{(m-1)}h_{i_{m-1}i_{m}}^{(m)}h_{j_{m-1}i_{m}}%
^{(m)}y_{i_{1}}^{(1)}\overline{y_{j_{1}}^{(1)}}\cdots y_{i_{m-1}}%
^{(m-1)}\overline{y_{j_{m-1}}^{(m-1)}}\right)  ^{1/2}\text{.}%
\end{align*}
Thus,%
\begin{align*}
&  \left\vert A_{0}\left(  x^{(1)},\dots,x^{(m)}\right)  \right\vert \\
&  \leq n_{m}^{1/2}\left(
%TCIMACRO{\tsum \limits_{\substack{i_{1},\dots,i_{m-1}=1\\j_{1},\dots
%,j_{m-1}=1}}^{n_{m}}}%
%BeginExpansion
{\textstyle\sum\limits_{\substack{i_{1},\dots,i_{m-1}=1\\j_{1},\dots
,j_{m-1}=1}}^{n_{m}}}
%EndExpansion
h_{i_{1}i_{2}}^{(2)}h_{j_{1}j_{2}}^{(2)}\cdots h_{i_{m-2}i_{m-1}}%
^{(m-1)}h_{j_{m-2}j_{m-1}}^{(m-1)}y_{i_{1}}^{(1)}\overline{y_{j_{1}}^{(1)}%
}\cdots y_{i_{m-1}}^{(m-1)}\overline{y_{j_{m-1}}^{(m-1)}}%
%TCIMACRO{\tsum \limits_{i_{m}=1}^{n_{m}}}%
%BeginExpansion
{\textstyle\sum\limits_{i_{m}=1}^{n_{m}}}
%EndExpansion
h_{i_{m-1}i_{m}}^{(m)}h_{j_{m-1}i_{m}}^{(m)}\right)  ^{1/2}\\
&  =n_{m}^{1/2}\left(
%TCIMACRO{\tsum \limits_{\substack{i_{1},\dots,i_{m-1}=1\\j_{1},\dots
%,j_{m-1}=1}}^{n_{m}}}%
%BeginExpansion
{\textstyle\sum\limits_{\substack{i_{1},\dots,i_{m-1}=1\\j_{1},\dots
,j_{m-1}=1}}^{n_{m}}}
%EndExpansion
h_{i_{1}i_{2}}^{(2)}h_{j_{1}j_{2}}^{(2)}\cdots h_{i_{m-2}i_{m-1}}%
^{(m-1)}h_{j_{m-2}j_{m-1}}^{(m-1)}y_{i_{1}}^{(1)}\overline{y_{j_{1}}^{(1)}%
}\cdots y_{i_{m-1}}^{(m-1)}\overline{y_{j_{m-1}}^{(m-1)}}\left\langle
u_{i_{m-1}}^{\left(  m\right)  },u_{j_{m-1}}^{\left(  m\right)  }\right\rangle
\right)  ^{1/2}%
\end{align*}
and, by (\ref{M1}), we have
\begin{align*}
&  \left\vert A_{0}\left(  x^{(1)},\dots,x^{(m)}\right)  \right\vert \\
&  \leq n_{m}^{1/2}\left(
%TCIMACRO{\tsum \limits_{i_{m-1}=1}^{n_{m}}}%
%BeginExpansion
{\textstyle\sum\limits_{i_{m-1}=1}^{n_{m}}}
%EndExpansion%
%TCIMACRO{\tsum \limits_{\substack{i_{1},\dots,i_{m-2}=1\\j_{1},\dots
%,j_{m-2}=1}}^{n_{m}}}%
%BeginExpansion
{\textstyle\sum\limits_{\substack{i_{1},\dots,i_{m-2}=1\\j_{1},\dots
,j_{m-2}=1}}^{n_{m}}}
%EndExpansion
h_{i_{1}i_{2}}^{(2)}h_{j_{1}j_{2}}^{(2)}\cdots h_{i_{m-2}i_{m-1}}%
^{(m-1)}h_{j_{m-2}i_{m-1}}^{(m-1)}y_{i_{1}}^{(1)}\overline{y_{j_{1}}^{(1)}%
}\cdots y_{i_{m-2}}^{(m-2)}\overline{y_{j_{m-2}}^{(m-2)}}\left\vert
y_{i_{m-1}}^{(m-1)}\right\vert ^{2}n_{m}\right)  ^{1/2}\\
&  \leq n_{m}\left(
%TCIMACRO{\tsum \limits_{\substack{i_{1},\dots,i_{m-2}=1\\j_{1},\dots
%,j_{m-2}=1}}^{n_{m}}}%
%BeginExpansion
{\textstyle\sum\limits_{\substack{i_{1},\dots,i_{m-2}=1\\j_{1},\dots
,j_{m-2}=1}}^{n_{m}}}
%EndExpansion
h_{i_{1}i_{2}}^{(2)}h_{j_{1}j_{2}}^{(2)}\cdots h_{i_{m-3}i_{m-2}}%
^{(m-2)}h_{j_{m-3}j_{m-2}}^{(m-2)}y_{i_{1}}^{(1)}\overline{y_{j_{1}}^{(1)}%
}\cdots y_{i_{m-2}}^{(m-2)}\overline{y_{j_{m-2}}^{(m-2)}}%
%TCIMACRO{\tsum \limits_{i_{m-1}=1}^{n_{m}}}%
%BeginExpansion
{\textstyle\sum\limits_{i_{m-1}=1}^{n_{m}}}
%EndExpansion
h_{i_{m-2}i_{m-1}}^{(m-1)}h_{j_{m-2}i_{m-1}}^{(m-1)}\right)  ^{1/2}\\
&  =n_{m}\left(
%TCIMACRO{\tsum \limits_{\substack{i_{1},\dots,i_{m-2}=1\\j_{1},\dots
%,j_{m-2}=1}}^{n_{m}}}%
%BeginExpansion
{\textstyle\sum\limits_{\substack{i_{1},\dots,i_{m-2}=1\\j_{1},\dots
,j_{m-2}=1}}^{n_{m}}}
%EndExpansion
h_{i_{1}i_{2}}^{(2)}h_{j_{1}j_{2}}^{(2)}\cdots h_{i_{m-3}i_{m-2}}%
^{(m-2)}h_{j_{m-3}j_{m-2}}^{(m-2)}y_{i_{1}}^{(1)}\overline{y_{j_{1}}^{(1)}%
}\cdots y_{i_{m-2}}^{(m-2)}\overline{y_{j_{m-2}}^{(m-2)}}\left\langle
u_{i_{m-2}}^{\left(  m-1\right)  },u_{j_{m-2}}^{\left(  m-1\right)
}\right\rangle \right)  ^{1/2}\text{.}%
\end{align*}
Since%
\begin{align*}
&  \left(
%TCIMACRO{\tsum \limits_{\substack{i_{1},\dots,i_{m-2}=1\\j_{1},\dots
%,j_{m-2}=1}}^{n_{m}}}%
%BeginExpansion
{\textstyle\sum\limits_{\substack{i_{1},\dots,i_{m-2}=1\\j_{1},\dots
,j_{m-2}=1}}^{n_{m}}}
%EndExpansion
h_{i_{1}i_{2}}^{(2)}h_{j_{1}j_{2}}^{(2)}\cdots h_{i_{m-3}i_{m-2}}%
^{(m-2)}h_{j_{m-3}j_{m-2}}^{(m-2)}y_{i_{1}}^{(1)}\overline{y_{j_{1}}^{(1)}%
}\cdots y_{i_{m-2}}^{(m-2)}\overline{y_{j_{m-2}}^{(m-2)}}\left\langle
u_{i_{m-2}}^{\left(  m-1\right)  },u_{j_{m-2}}^{\left(  m-1\right)
}\right\rangle \right)  ^{1/2}\\
&  =n_{m-1}^{1/2}\left(
%TCIMACRO{\tsum \limits_{i_{m-2}=1}^{n_{m}}}%
%BeginExpansion
{\textstyle\sum\limits_{i_{m-2}=1}^{n_{m}}}
%EndExpansion%
%TCIMACRO{\tsum \limits_{\substack{i_{1},\dots,i_{m-3}=1\\j_{1},\dots
%,j_{m-3}=1}}^{n_{m}}}%
%BeginExpansion
{\textstyle\sum\limits_{\substack{i_{1},\dots,i_{m-3}=1\\j_{1},\dots
,j_{m-3}=1}}^{n_{m}}}
%EndExpansion
h_{i_{1}i_{2}}^{(2)}h_{j_{1}j_{2}}^{(2)}\cdots h_{i_{m-3}i_{m-2}}%
^{(m-2)}h_{j_{m-3}i_{m-2}}^{(m-2)}y_{i_{1}}^{(1)}\overline{y_{j_{1}}^{(1)}%
}\cdots y_{i_{m-3}}^{(m-3)}\overline{y_{j_{m-3}}^{(m-3)}}y_{i_{m-2}}%
^{(m-2)}\overline{y_{i_{m-2}}^{(m-2)}}\right)  ^{1/2}\\
&  =n_{m-1}^{1/2}\left(
%TCIMACRO{\tsum \limits_{i_{m-2}=1}^{n_{m}}}%
%BeginExpansion
{\textstyle\sum\limits_{i_{m-2}=1}^{n_{m}}}
%EndExpansion
|y_{i_{m-2}}^{(m-2)}|^{2}%
%TCIMACRO{\tsum \limits_{\substack{i_{1},\dots,i_{m-3}=1\\j_{1},\dots
%,j_{m-3}=1}}^{n_{m}}}%
%BeginExpansion
{\textstyle\sum\limits_{\substack{i_{1},\dots,i_{m-3}=1\\j_{1},\dots
,j_{m-3}=1}}^{n_{m}}}
%EndExpansion
h_{i_{1}i_{2}}^{(2)}h_{j_{1}j_{2}}^{(2)}\cdots h_{i_{m-3}i_{m-2}}%
^{(m-2)}h_{j_{m-3}i_{m-2}}^{(m-2)}y_{i_{1}}^{(1)}\overline{y_{j_{1}}^{(1)}%
}\cdots y_{i_{m-3}}^{(m-3)}\overline{y_{j_{m-3}}^{(m-3)}}\right)
^{1/2}\text{,}%
\end{align*}
we have%
\begin{align*}
&  \left\vert A_{0}\left(  x^{(1)},\dots,x^{(m)}\right)  \right\vert \\
&  \leq n_{m}n_{m-1}^{1/2}\left(
%TCIMACRO{\tsum \limits_{i_{m-2}=1}^{n_{m}}}%
%BeginExpansion
{\textstyle\sum\limits_{i_{m-2}=1}^{n_{m}}}
%EndExpansion%
%TCIMACRO{\tsum \limits_{\substack{i_{1},\dots,i_{m-3}=1\\j_{1},\dots
%,j_{m-3}=1}}^{n_{m}}}%
%BeginExpansion
{\textstyle\sum\limits_{\substack{i_{1},\dots,i_{m-3}=1\\j_{1},\dots
,j_{m-3}=1}}^{n_{m}}}
%EndExpansion
h_{i_{1}i_{2}}^{(2)}h_{j_{1}j_{2}}^{(2)}\cdots h_{i_{m-3}i_{m-2}}%
^{(m-2)}h_{j_{m-3}i_{m-2}}^{(m-2)}y_{i_{1}}^{(1)}\overline{y_{j_{1}}^{(1)}%
}\cdots y_{i_{m-3}}^{(m-3)}\overline{y_{j_{m-3}}^{(m-3)}}\right)  ^{1/2}%
\end{align*}
and, repeating this procedure, we finally obtain%
\begin{align*}
\left\vert A_{0}\left(  x^{(1)},\dots,x^{(m)}\right)  \right\vert  &  \leq
n_{m}n_{m-1}^{1/2}\cdots n_{2}^{1/2}\left(
%TCIMACRO{\tsum \limits_{i_{1}=1}^{n_{m}}}%
%BeginExpansion
{\textstyle\sum\limits_{i_{1}=1}^{n_{m}}}
%EndExpansion
\left\vert y_{i_{1}}^{(1)}\right\vert ^{2}\right)  ^{1/2}\\
&  \leq n_{m}^{1/2}%
%TCIMACRO{\tprod \limits_{j=1}^{m}}%
%BeginExpansion
{\textstyle\prod\limits_{j=1}^{m}}
%EndExpansion
n_{j}^{1/2}\text{.}%
\end{align*}

\end{proof}

Now we are able to prove Theorem \ref{090}.

\begin{proof}
[Proof of Theorem \ref{090}]It sufficies to consider $\mathbb{K}=\mathbb{C}$;
the case of real scalars is a straightforward consequence. We shall divide the
proof into three steps. Let $n_{1},\ldots,n_{m}$ be positive integers.

\bigskip

\noindent\textbf{First Step. }(See \cite[p. 295]{JSpencer}) Given $\delta>0$,
there exists a positive integer $N$ such that, for all $n_{k}>N$ there is a
Hadamard order $t_{k}$ such that
\begin{equation}
n_{k}\leq t_{k}<\left(  1+\delta\right)  n_{k}\text{, \ \ \ \ }k=1,\ldots
,m\text{.} \label{M2}%
\end{equation}
As commented in \cite[p. 295]{JSpencer}, it suffices to observe that there are
Hadamard matrices of order $4^{i}12^{j}$ for all $i,j$. Since the sequence
defined recursively by%
\[
x_{1}=1\text{ \ \ \ \ and \ \ \ \ }x_{n+1}=\min\left\{  \left\{  4^{i}%
12^{j}:\left(  i,j\right)  \in\left\{  0,1,2,\ldots\right\}  \times\left\{
0,1,2,\ldots\right\}  \right\}  \backslash\left\{  x_{1},\ldots,x_{n}\right\}
\right\}
\]
is such that $\lim\limits_{n\rightarrow\infty}\dfrac{x_{n+1}}{x_{n}}=1$, a
straightforward calculation assures (\ref{M2}).%

%TCIMACRO{\TeXButton{Figura}{\begin{figure}[H]
%\centering\includegraphics[scale=0.8]{grafico.pdf}\caption{Graphic of
%$\frac{x_{n+1}}{x_{n}}$}
%\end{figure}}}%
%BeginExpansion
\begin{figure}[H]
\centering\includegraphics[scale=0.8]{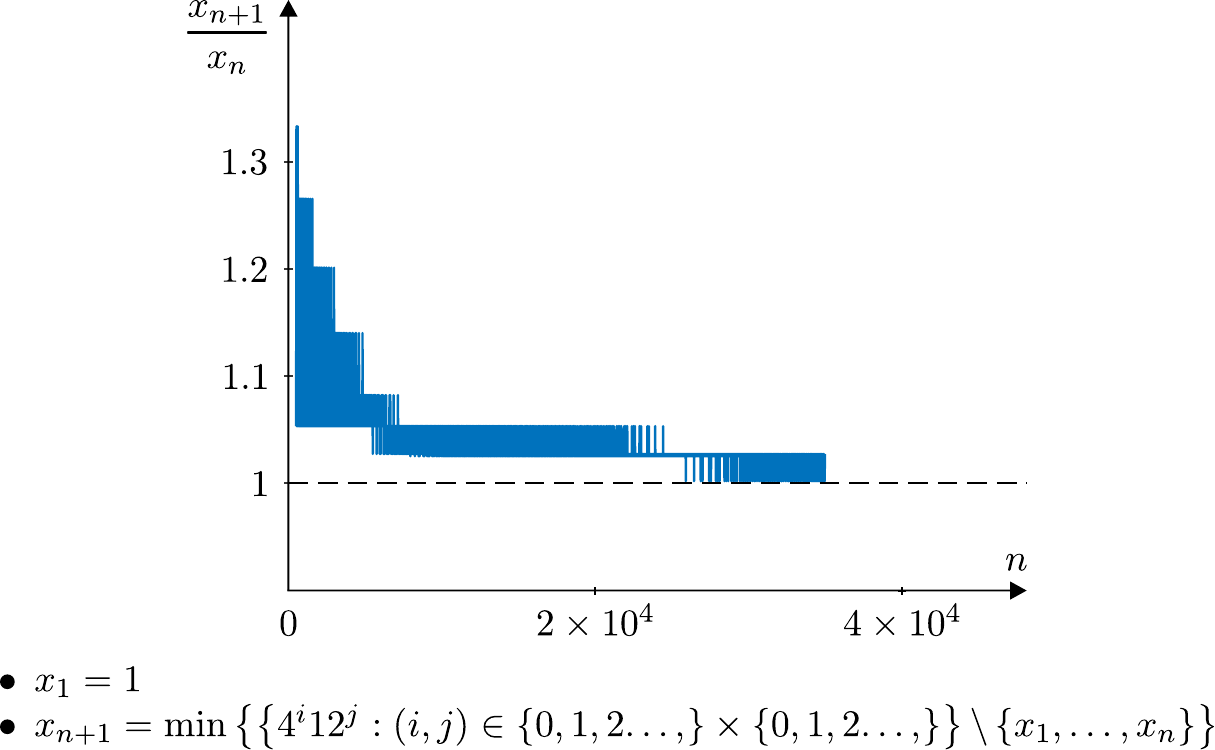}\caption{Graphic of
$\frac{x_{n+1}}{x_{n}}$}
\end{figure}%
%EndExpansion

\bigskip

\noindent\textbf{Second Step}. Notice that we can assume, without loss of
generality, $t_{1}\leq\cdots\leq t_{m}$. It follows from Lemma \ref{LemaM1}
that, for the integers $t_{1},\ldots,t_{m}$ of the first step, there exists an
$m$-linear form $A_{0}\colon\ell_{\infty}^{t_{1}}\times\cdots\times
\ell_{\infty}^{t_{m}}\longrightarrow\mathbb{C}$ with coefficients $\pm1$ such
that%
\[
\left\Vert A_{0}\right\Vert \leq t_{m}^{1/2}%
%TCIMACRO{\tprod \limits_{k=1}^{m}}%
%BeginExpansion
{\textstyle\prod\limits_{k=1}^{m}}
%EndExpansion
t_{k}^{1/2}\text{.}%
\]

\bigskip

\noindent\textbf{Third Step. }If $n_{k}=t_{k}$ for all $k$, we just consider
$A_{n_{1},\ldots,n_{m}}=A_{0}$. Otherwise, we define%
\[
A_{n_{1},\ldots,n_{m}}\colon\ell_{\infty}^{n_{1}}\times\cdots\times
\ell_{\infty}^{n_{m}}\longrightarrow\mathbb{C}%
\]
by%
\[
A_{n_{1},\ldots,n_{m}}\left(  \left(  z_{i_{1}}^{\left(  1\right)  }\right)
_{i_{1}=1}^{n_{1}},\ldots,\left(  z_{i_{m}}^{\left(  m\right)  }\right)
_{i_{m}=1}^{n_{m}}\right)  =A_{0}\left(  \left(  z_{1}^{\left(  1\right)
},\ldots,z_{n_{1}}^{\left(  1\right)  },0,\ldots,0\right)  ,\ldots,\left(
z_{1}^{\left(  m\right)  },\ldots,z_{n_{m}}^{\left(  m\right)  }%
,0,\ldots,0\right)  \right)
\]
Then, given $x^{\left(  k\right)  }\in B_{\ell_{\infty}^{n_{k}}}$,
$k=1,\ldots,m$, we have%
\begin{align*}
&  \left\vert A_{n_{1},\ldots,n_{m}}\left(  \left(  x_{i_{1}}^{\left(
1\right)  }\right)  _{i_{1}=1}^{n_{1}},\ldots,\left(  x_{i_{m}}^{\left(
m\right)  }\right)  _{i_{m}=1}^{n_{m}}\right)  \right\vert \\
&  =\left\vert A_{0}\left(  \left(  x_{1}^{\left(  1\right)  },\ldots
,x_{n_{1}}^{\left(  1\right)  },0,\ldots,0\right)  ,\ldots,\left(
x_{1}^{\left(  m\right)  },\ldots,x_{n_{m}}^{\left(  m\right)  }%
,0,\ldots,0\right)  \right)  \right\vert \\
&  \leq\left\Vert A_{0}\right\Vert \leq t_{m}^{1/2}%
%TCIMACRO{\tprod \limits_{k=1}^{m}}%
%BeginExpansion
{\textstyle\prod\limits_{k=1}^{m}}
%EndExpansion
t_{k}^{1/2}\\
&  \leq\left(  1+\delta\right)  ^{\frac{m+1}{2}}\max\left\{  n_{1}%
^{1/2},\ldots,n_{m}^{1/2}\right\}
%TCIMACRO{\tprod \limits_{j=1}^{m}}%
%BeginExpansion
{\textstyle\prod\limits_{j=1}^{m}}
%EndExpansion
n_{j}^{1/2}%
\end{align*}
and this ends the proof.
\end{proof}

\begin{remark}
A famous conjecture of Hadamard proposes that there exists a Hadamard matrix
of order $4k$ for every positive integer $k$. If this is true or, at least, if
the gap between two consecutive orders of Hadamard matrices is bounded by a
universal constant, a straightforward modification of our proof shows that
there exists an $m$-linear form $A_{n}\colon\ell_{\infty}^{n}\times
\cdots\times\ell_{\infty}^{n}\longrightarrow\mathbb{K}$ of the type
\[
A_{n}(z^{(1)},\ldots,z^{(m)})=%
%TCIMACRO{\tsum \limits_{j_{1}=1}^{n}}%
%BeginExpansion
{\textstyle\sum\limits_{j_{1}=1}^{n}}
%EndExpansion
\cdots%
%TCIMACRO{\tsum \limits_{j_{m}=1}^{n}}%
%BeginExpansion
{\textstyle\sum\limits_{j_{m}=1}^{n}}
%EndExpansion
\pm z_{j_{1}}^{(1)}\cdots z_{j_{m}}^{(m)}\text{,}%
\]
such that%
\[
\left\Vert A_{n}\right\Vert \leq\left(  1+O\left(  n^{-1/2}\right)  \right)
n^{\frac{m+1}{2}}\text{.}%
\]

\end{remark}

\begin{corollary}
Let a positive integer $m$ and $\epsilon>0$ be given. There exists a positive
integer $N$ such that, for all $n>N$, there exists an $m$-linear form
$A_{n}\colon\ell_{\infty}^{n}\times\cdots\times\ell_{\infty}^{n}%
\longrightarrow\mathbb{K}$ of the type
\[
A_{n}(z^{(1)},\ldots,z^{(m)})=%
%TCIMACRO{\tsum \limits_{j_{1}=1}^{n}}%
%BeginExpansion
{\textstyle\sum\limits_{j_{1}=1}^{n}}
%EndExpansion
\cdots%
%TCIMACRO{\tsum \limits_{j_{m}=1}^{n}}%
%BeginExpansion
{\textstyle\sum\limits_{j_{m}=1}^{n}}
%EndExpansion
\pm z_{j_{1}}^{(1)}\cdots z_{j_{m}}^{(m)}\text{,}%
\]
satisfying
\[
\Vert A_{n}\Vert\leq\left(  1+\epsilon\right)  n^{\frac{m+1}{2}}\text{.}%
\]

\end{corollary}

\section{Constants of Bennett's inequality}

In this section we shall show that, in most of the cases (see Figure 2), the
constants of Bennett's inequality (\ref{I3}) are \textquotedblleft
uniformly\textquotedblright\ asymptotically bounded by $1$ and, besides, the
constant $1$ will be shown to be sharp in some cases (see Corollary
\ref{op71}). We begin with a multipurpose simple lemma that shall be also used
in the next section. For the sake of simplicity, we use the notation
$\mathbb{N}_{k}$ to represent the set $\{k,k+1,k+2,\ldots\}.$

\begin{lemma}
\label{LemaB1}Let $k,m\in\mathbb{N}$ and let $f\colon\mathbb{N}_{k}%
\longrightarrow\lbrack0,\infty)$ be an increasing function. If for all
$n\in\mathbb{N}_{k}$ there exists an $m$-linear form $A_{0}\colon\ell_{2}%
^{n}\times\cdots\times\ell_{2}^{n}\longrightarrow\mathbb{K}$ with coefficients
$\pm1$ such that%
\[
\left\Vert A_{0}\right\Vert \leq f\left(  n\right)  \text{,}%
\]
then, for all $p_{1},\ldots,p_{m}\in\left[  2,\infty\right]  $ and all
$n_{1},\ldots,n_{m}\in\mathbb{N}_{k}$ there is an $m$-linear form $A\colon
\ell_{p_{1}}^{n_{1}}\times\cdots\times\ell_{p_{m}}^{n_{m}}\longrightarrow
\mathbb{K}$ with coefficients $\pm1$ such that%
\[
\left\Vert A\right\Vert \leq\max\left\{  f\left(  n_{1}\right)  ,\ldots
,f\left(  n_{m}\right)  \right\}  n_{1}^{\frac{1}{2}-\frac{1}{p_{1}}}\cdots
n_{m}^{\frac{1}{2}-\frac{1}{p_{m}}}\text{.}%
\]

\end{lemma}

\begin{proof}
Let $n=\max\left\{  n_{1},\ldots,n_{m}\right\}  $. From our hypothesis, there
is an $m$-linear form $A_{0}\colon\ell_{2}^{n}\times\cdots\times\ell_{2}%
^{n}\longrightarrow\mathbb{K}$ of the type%
\[
A_{0}\left(  x^{\left(  1\right)  },\ldots,x^{\left(  m\right)  }\right)  =%
%TCIMACRO{\tsum \limits_{i_{1}=1}^{n}}%
%BeginExpansion
{\textstyle\sum\limits_{i_{1}=1}^{n}}
%EndExpansion
\cdots%
%TCIMACRO{\tsum \limits_{i_{m}=1}^{n}}%
%BeginExpansion
{\textstyle\sum\limits_{i_{m}=1}^{n}}
%EndExpansion
\pm x_{i_{1}}^{\left(  1\right)  }\cdots x_{i_{m}}^{\left(  m\right)  }%
\]
such that%
\[
\left\Vert A_{0}\right\Vert \leq f\left(  n\right)  \text{.}%
\]
Let us consider the $m$-linear form $A_{1}\colon\ell_{2}^{n_{1}}\times
\cdots\times\ell_{2}^{n_{m}}\longrightarrow\mathbb{K}$ defined by%
\[
A_{1}\left(  x^{\left(  1\right)  },\ldots,x^{\left(  m\right)  }\right)
=A_{0}\left(  \left(  x_{1}^{\left(  1\right)  },\ldots,x_{n_{1}}^{\left(
1\right)  },0,\ldots,0\right)  ,\ldots,\left(  x_{1}^{\left(  m\right)
},\ldots,x_{n_{m}}^{\left(  m\right)  },0,\ldots,0\right)  \right)  \text{.}%
\]
It is plain that%
\begin{equation}
\left\Vert A_{1}\right\Vert \leq\left\Vert A_{0}\right\Vert \leq f\left(
n\right)  =\max\left\{  f\left(  n_{1}\right)  ,\ldots,f\left(  n_{m}\right)
\right\}  \text{.} \label{B1}%
\end{equation}
Now, let us define $A\colon\ell_{p_{1}}^{n_{1}}\times\cdots\times\ell_{p_{m}%
}^{n_{m}}\longrightarrow\mathbb{K}$ by%
\[
A\left(  x^{\left(  1\right)  },\ldots,x^{\left(  m\right)  }\right)
=A_{1}\left(  x^{\left(  1\right)  },\ldots,x^{\left(  m\right)  }\right)
\text{.}%
\]
In this case,%
\begin{align}
\left\Vert A\right\Vert  &  =\sup_{\left\Vert x^{\left(  k\right)
}\right\Vert _{p_{k}}\leq1}\left\vert
%TCIMACRO{\tsum \limits_{i_{1}=1}^{n_{1}}}%
%BeginExpansion
{\textstyle\sum\limits_{i_{1}=1}^{n_{1}}}
%EndExpansion
\cdots%
%TCIMACRO{\tsum \limits_{i_{m}=1}^{n_{m}}}%
%BeginExpansion
{\textstyle\sum\limits_{i_{m}=1}^{n_{m}}}
%EndExpansion
\pm x_{i_{1}}^{\left(  1\right)  }\cdots x_{i_{m}}^{\left(  m\right)
}\right\vert \nonumber\\
&  =n_{1}^{\frac{1}{2}-\frac{1}{p_{1}}}\cdots n_{m}^{\frac{1}{2}-\frac
{1}{p_{m}}}\sup_{\left\Vert x^{\left(  k\right)  }\right\Vert _{p_{k}}\leq
1}\left\vert
%TCIMACRO{\tsum \limits_{i_{1}=1}^{n_{1}}}%
%BeginExpansion
{\textstyle\sum\limits_{i_{1}=1}^{n_{1}}}
%EndExpansion
\cdots%
%TCIMACRO{\tsum \limits_{i_{m}=1}^{n_{m}}}%
%BeginExpansion
{\textstyle\sum\limits_{i_{m}=1}^{n_{m}}}
%EndExpansion
\pm\frac{x_{i_{1}}^{\left(  1\right)  }}{n_{1}^{\frac{1}{2}-\frac{1}{p_{1}}}%
}\cdots\frac{x_{i_{m}}^{\left(  m\right)  }}{n_{m}^{\frac{1}{2}-\frac{1}%
{p_{m}}}}\right\vert \text{.} \label{B2}%
\end{align}
For all $k=1,\ldots,m$, let $z^{\left(  k\right)  }=n_{k}^{\frac{1}{p_{k}%
}-\frac{1}{2}}x^{\left(  k\right)  }$. By the H\"{o}lder inequality, since
$\left\Vert x^{\left(  k\right)  }\right\Vert _{p_{k}}\leq1$, we have
\[
\left\Vert z^{\left(  k\right)  }\right\Vert _{2}\leq1\text{.}%
\]
It follows from (\ref{B1}) and (\ref{B2}) that
\begin{align*}
\left\Vert A\right\Vert  &  \leq n_{1}^{\frac{1}{2}-\frac{1}{p_{1}}}\cdots
n_{m}^{\frac{1}{2}-\frac{1}{p_{m}}}\sup_{\left\Vert z^{\left(  k\right)
}\right\Vert _{2}\leq1}\left\vert
%TCIMACRO{\tsum \limits_{i_{1}=1}^{n_{1}}}%
%BeginExpansion
{\textstyle\sum\limits_{i_{1}=1}^{n_{1}}}
%EndExpansion
\cdots%
%TCIMACRO{\tsum \limits_{i_{m}=1}^{n_{m}}}%
%BeginExpansion
{\textstyle\sum\limits_{i_{m}=1}^{n_{m}}}
%EndExpansion
\pm z_{i_{1}}^{\left(  1\right)  }\cdots z_{i_{m}}^{\left(  m\right)
}\right\vert \\
&  =\left\Vert A_{1}\right\Vert n_{1}^{\frac{1}{2}-\frac{1}{p_{1}}}\cdots
n_{m}^{\frac{1}{2}-\frac{1}{p_{m}}}\\
&  \leq\max\left\{  f\left(  n_{1}\right)  ,\ldots,f\left(  n_{m}\right)
\right\}  n_{1}^{\frac{1}{2}-\frac{1}{p_{1}}}\cdots n_{m}^{\frac{1}{2}%
-\frac{1}{p_{m}}}\text{.}%
\end{align*}

\end{proof}

Now we prove another lemma which plays a fundamental role in this section.

\begin{lemma}
Let $\epsilon>0$. There exists a positive integer $N$ $\mathrm{(}$depending
just on $\epsilon\mathrm{)}$ such that, whenever $n>N$, there is a bilinear
form $A\colon\ell_{2}^{n}\times\ell_{2}^{n}\longrightarrow\mathbb{K}$ of the
type
\[
A\left(  z^{\left(  1\right)  },z^{\left(  2\right)  }\right)  =%
%TCIMACRO{\tsum \limits_{i=1}^{n}}%
%BeginExpansion
{\textstyle\sum\limits_{i=1}^{n}}
%EndExpansion%
%TCIMACRO{\tsum \limits_{j=1}^{n}}%
%BeginExpansion
{\textstyle\sum\limits_{j=1}^{n}}
%EndExpansion
\pm z_{i}^{\left(  1\right)  }z_{j}^{\left(  2\right)  }\text{,}%
\]
such that%
\[
\Vert A\Vert\leq\left(  1+\epsilon\right)  n^{1/2}\text{.}%
\]

\end{lemma}

\begin{proof}
It suffices to consider $\mathbb{K}=\mathbb{C}$. Let $\left[  h_{ij}\right]
_{n\times n}$ be a Hadamard matrix of order $n$. It is simple to show that the
bilinear form $A_{0}\colon\ell_{2}^{n}\times\ell_{2}^{n}\longrightarrow
\mathbb{C}$ given by%
\begin{equation}
A_{0}(z^{(1)},z^{(2)})=%
%TCIMACRO{\tsum \limits_{i=1}^{n}}%
%BeginExpansion
{\textstyle\sum\limits_{i=1}^{n}}
%EndExpansion%
%TCIMACRO{\tsum \limits_{j=1}^{n}}%
%BeginExpansion
{\textstyle\sum\limits_{j=1}^{n}}
%EndExpansion
h_{ij}z_{i}^{(1)}z_{j}^{(2)}\text{,} \label{B3}%
\end{equation}
has norm%
\[
\Vert A_{0}\Vert\leq n^{1/2}\text{.}%
\]
In fact, if $x^{\left(  1\right)  },x^{\left(  2\right)  }\in B_{\ell_{2}^{n}%
}$, by the Cauchy-Schwarz inequality, we have%
\begin{align*}
\left\vert A_{0}\left(  x^{\left(  1\right)  },x^{\left(  2\right)  }\right)
\right\vert  &  \leq%
%TCIMACRO{\tsum \limits_{j=1}^{n}}%
%BeginExpansion
{\textstyle\sum\limits_{j=1}^{n}}
%EndExpansion
\left\vert
%TCIMACRO{\tsum \limits_{i=1}^{n}}%
%BeginExpansion
{\textstyle\sum\limits_{i=1}^{n}}
%EndExpansion
h_{ij}x_{i}^{(1)}\right\vert \left\vert x_{j}^{\left(  2\right)  }\right\vert
\\
&  \leq\left(
%TCIMACRO{\tsum \limits_{j=1}^{n}}%
%BeginExpansion
{\textstyle\sum\limits_{j=1}^{n}}
%EndExpansion
\left\vert x_{j}^{(2)}\right\vert ^{2}\right)  ^{1/2}\cdot\left(
%TCIMACRO{\tsum \limits_{j=1}^{n}}%
%BeginExpansion
{\textstyle\sum\limits_{j=1}^{n}}
%EndExpansion
\left\vert
%TCIMACRO{\tsum \limits_{i=1}^{n}}%
%BeginExpansion
{\textstyle\sum\limits_{i=1}^{n}}
%EndExpansion
h_{ij}x_{i}^{(1)}\right\vert ^{2}\right)  ^{1/2}\\
&  \leq\left(
%TCIMACRO{\tsum \limits_{j=1}^{n}}%
%BeginExpansion
{\textstyle\sum\limits_{j=1}^{n}}
%EndExpansion
\,%
%TCIMACRO{\tsum \limits_{i,k=1}^{n}}%
%BeginExpansion
{\textstyle\sum\limits_{i,k=1}^{n}}
%EndExpansion
h_{ij}h_{kj}x_{i}^{(1)}\overline{x_{k}^{(1)}}\right)  ^{1/2}\\
&  =\left(
%TCIMACRO{\tsum \limits_{i,k=1}^{n}}%
%BeginExpansion
{\textstyle\sum\limits_{i,k=1}^{n}}
%EndExpansion
x_{i}^{(1)}\overline{x_{k}^{(1)}}n\delta_{ik}\right)  ^{1/2}\\
&  \leq n^{1/2}\text{.}%
\end{align*}
We shall show that for other values of $n$, we have the same inequality, just
with an extra multiplicative \textquotedblleft asymptotic
factor\textquotedblright\ $\left(  1+\epsilon\right)  $.

As in the first step of the proof of Theorem \ref{090}, given $\delta>0$,
there exists a positive integer $N$ such that, whenever $n>N$, there is a
Hadamard matrix of order $t$ satisfying
\begin{equation}
n\leq t\leq n\left(  1+\delta\right)  \text{.} \label{B4}%
\end{equation}
Now, let us consider a Hadamard matrix $\left[  h_{ij}\right]  _{t\times t}$
of order $t$. Let $A_{0}\colon\ell_{2}^{t}\times\ell_{2}^{t}\longrightarrow
\mathbb{C}$ be as in (\ref{B3}). Thus,%
\[
\left\Vert A_{0}\right\Vert \leq t^{1/2}\text{.}%
\]

If $n=t$, we just consider $A=A_{0}$. If $n<t$, we define%
\[
A\colon\ell_{2}^{n}\times\ell_{2}^{n}\longrightarrow\mathbb{C}\text{,}%
\]
by
\[
A\left(  \left(  z_{i}^{\left(  1\right)  }\right)  _{i=1}^{n},\left(
z_{j}^{\left(  2\right)  }\right)  _{j=1}^{n}\right)  =A_{0}\left(  \left(
z_{1}^{\left(  1\right)  },\ldots,z_{n}^{\left(  1\right)  },0,\ldots
,0\right)  ,\left(  z_{1}^{\left(  2\right)  },\ldots,z_{n}^{\left(  2\right)
},0,\ldots,0\right)  \right)  \text{.}%
\]
Then, given $x^{\left(  1\right)  },x^{\left(  2\right)  }\in B_{\ell_{2}^{n}%
}$, by (\ref{B4}),%
\begin{align*}
\left\vert A\left(  \left(  x_{i}^{\left(  1\right)  }\right)  _{i=1}%
^{n},\left(  x_{j}^{\left(  2\right)  }\right)  _{j=1}^{n}\right)
\right\vert  &  =\left\vert A_{0}\left(  \left(  x_{1}^{\left(  1\right)
},\ldots,x_{n}^{\left(  1\right)  },0,\ldots,0\right)  ,\left(  x_{1}^{\left(
2\right)  },\ldots,x_{n}^{\left(  2\right)  },0,\ldots,0\right)  \right)
\right\vert \\
&  \leq\left\Vert A_{0}\right\Vert \leq t^{1/2}\leq\left(  1+\delta\right)
^{1/2}n^{1/2}%
\end{align*}
and the proof is complete.
\end{proof}

Combining the above result with Lemma \ref{LemaB1} we obtain the main result
of this section:

\begin{theorem}
\label{TeoB3}Let $\epsilon>0$ and $p_{1},p_{2}\in\lbrack2,\infty]$. There
exists a positive integer $N$ $\mathrm{(}$depending just on $\epsilon
\mathrm{)}$ such that, whenever $\max\left\{  n_{1},n_{2}\right\}  >N$, there
is a bilinear form $A\colon\ell_{p_{1}}^{n_{1}}\times\ell_{p_{2}}^{n_{2}%
}\longrightarrow\mathbb{K}$ of the type
\[
A(z^{(1)},z^{(2)})=%
%TCIMACRO{\tsum \limits_{i=1}^{n_{1}}}%
%BeginExpansion
{\textstyle\sum\limits_{i=1}^{n_{1}}}
%EndExpansion%
%TCIMACRO{\tsum \limits_{j=1}^{n_{2}}}%
%BeginExpansion
{\textstyle\sum\limits_{j=1}^{n_{2}}}
%EndExpansion
\pm z_{i}^{(1)}z_{j}^{(2)}\text{,}%
\]
such that%
\begin{equation}
\Vert A\Vert\leq\left(  1+\epsilon\right)  \max\left\{  n_{2}^{1/p_{2}^{\ast}%
}n_{1}^{\max\left\{  \frac{1}{2}-\frac{1}{p_{1}},0\right\}  },n_{1}%
^{1/p_{1}^{\ast}}n_{2}^{\max\left\{  \frac{1}{2}-\frac{1}{p_{2}},0\right\}
}\right\}  \text{.} \label{B5}%
\end{equation}

\end{theorem}

Now we prove three corollaries that essentially cover the remaining cases (see
Figure \ref{Fig2}):

\begin{corollary}
\label{CorB4}Let $\epsilon>0$ and $1\leq p_{1}<2\leq p_{2}$ $\mathrm{(}1\leq
p_{2}<2\leq p_{1}\mathrm{)}$. There exists a positive integer $N$ $\mathrm{(}%
$depending just on $\epsilon\mathrm{)}$ such that, whenever $n_{2}>N$ and
$n_{1}\leq n_{2}$ $\mathrm{(}n_{1}\geq N$ and $n_{2}\leq n_{1}\mathrm{)}$,
there is a bilinear form $A\colon\ell_{p_{1}}^{n_{1}}\times\ell_{p_{2}}%
^{n_{2}}\longrightarrow\mathbb{K}$ of the type
\[
A(x,y)=%
%TCIMACRO{\tsum \limits_{i=1}^{n_{1}}}%
%BeginExpansion
{\textstyle\sum\limits_{i=1}^{n_{1}}}
%EndExpansion%
%TCIMACRO{\tsum \limits_{j=1}^{n_{2}}}%
%BeginExpansion
{\textstyle\sum\limits_{j=1}^{n_{2}}}
%EndExpansion
\pm x_{i}y_{j}\text{,}%
\]
such that%
\[
\Vert A\Vert\leq\left(  1+\epsilon\right)  \max\left\{  n_{2}^{1/p_{2}^{\ast}%
}n_{1}^{\max\left\{  \frac{1}{2}-\frac{1}{p_{1}},0\right\}  },n_{1}%
^{1/p_{1}^{\ast}}n_{2}^{\max\left\{  \frac{1}{2}-\frac{1}{p_{2}},0\right\}
}\right\}  \text{.}%
\]

\end{corollary}

\begin{proof}
Let us assume $1\leq p_{1}<2\leq p_{2}$ and $n_{1}\leq n_{2}$. Notice that in
this case%
\[
\max\left\{  n_{2}^{1/p_{2}^{\ast}}n_{1}^{\max\left\{  \frac{1}{2}-\frac
{1}{p_{1}},0\right\}  },n_{1}^{1/p_{1}^{\ast}}n_{2}^{\max\left\{  \frac{1}%
{2}-\frac{1}{p_{2}},0\right\}  }\right\}  =n_{2}^{1/p_{2}^{\ast}}\text{.}%
\]
By Theorem \ref{TeoB3}, given $\epsilon>0$, there exists $N$ (depending just
on $\epsilon$) such that, if $n_{2}>N$, there is a bilinear form $A_{1}%
\colon\ell_{2}^{n_{1}}\times\ell_{p_{2}}^{n_{2}}\longrightarrow\mathbb{K}$ of
the type%
\[
A_{1}(z^{\left(  1\right)  },z^{\left(  2\right)  })=%
%TCIMACRO{\tsum \limits_{i=1}^{n_{1}}}%
%BeginExpansion
{\textstyle\sum\limits_{i=1}^{n_{1}}}
%EndExpansion%
%TCIMACRO{\tsum \limits_{j=1}^{n_{2}}}%
%BeginExpansion
{\textstyle\sum\limits_{j=1}^{n_{2}}}
%EndExpansion
\pm z_{i}^{\left(  1\right)  }z_{j}^{\left(  2\right)  }\text{,}%
\]
satisfying%
\[
\Vert A_{1}\Vert\leq\left(  1+\epsilon\right)  n_{2}^{1/p_{2}^{\ast}}\text{.}%
\]
Now, define $A\colon\ell_{p_{1}}^{n_{1}}\times\ell_{p_{2}}^{n_{2}%
}\longrightarrow\mathbb{K}$ by $A\left(  x,y\right)  =A_{1}\left(  x,y\right)
$ and note that%
\[
\left\Vert A\right\Vert \leq\Vert A_{1}\Vert\leq\left(  1+\epsilon\right)
n_{2}^{1/p_{2}^{\ast}}\text{.}%
\]
The case $1\leq p_{2}<2\leq p_{1}$ and $n_{2}\leq n_{1}$ is analogous.
\end{proof}

\begin{corollary}
\label{CorB5}Let $\epsilon>0$ and $p\in\lbrack1,\infty]$. There exists a
positive integer $N$ $\mathrm{(}$depending just on $\epsilon\mathrm{)}$ such
that, whenever $\max\left\{  n_{1},n_{2}\right\}  >N$, there is a bilinear
form $A\colon\ell_{p}^{n_{1}}\times\ell_{p}^{n_{2}}\longrightarrow\mathbb{K}$
of the type
\[
A(z^{(1)},z^{(2)})=%
%TCIMACRO{\tsum \limits_{i=1}^{n_{1}}}%
%BeginExpansion
{\textstyle\sum\limits_{i=1}^{n_{1}}}
%EndExpansion%
%TCIMACRO{\tsum \limits_{j=1}^{n_{2}}}%
%BeginExpansion
{\textstyle\sum\limits_{j=1}^{n_{2}}}
%EndExpansion
\pm z_{i}^{(1)}z_{j}^{(2)}\text{,}%
\]
such that%
\[
\Vert A\Vert\leq\left(  1+\epsilon\right)  \max\left\{  n_{2}^{1/p^{\ast}%
}n_{1}^{\max\left\{  \frac{1}{2}-\frac{1}{p},0\right\}  },n_{1}^{1/p^{\ast}%
}n_{2}^{\max\left\{  \frac{1}{2}-\frac{1}{p},0\right\}  }\right\}  \text{.}%
\]

\end{corollary}

\begin{proof}
It is suffices to consider $p\in\left[  1,2\right)  $, because the case
$p\in\left[  2,\infty\right]  $ is solved in Theorem \ref{TeoB3}. If
$n_{1},n_{2}\in\mathbb{N}$ and $\varepsilon_{ij}\in\left\{  -1,1\right\}  $
for each $\left(  i,j\right)  \in\left\{  1,\ldots,n_{1}\right\}
\times\left\{  1,\ldots,n_{2}\right\}  $, it is obvious that all bilinear
forms $A\colon\ell_{1}^{n_{1}}\times\ell_{1}^{n_{2}}\longrightarrow\mathbb{C}$
defined by
\begin{equation}
A\left(  z^{\left(  1\right)  },z^{\left(  2\right)  }\right)  =%
%TCIMACRO{\tsum \limits_{i=1}^{n_{1}}}%
%BeginExpansion
{\textstyle\sum\limits_{i=1}^{n_{1}}}
%EndExpansion%
%TCIMACRO{\tsum \limits_{j=1}^{n_{2}}}%
%BeginExpansion
{\textstyle\sum\limits_{j=1}^{n_{2}}}
%EndExpansion
\varepsilon_{ij}z_{i}^{\left(  1\right)  }z_{j}^{\left(  2\right)  }
\label{B6}%
\end{equation}
have norm precisely $1$. If $1\leq p<2$, by the Riesz-Thorin Theorem applied
to (\ref{B6}) and (\ref{B5}) with $p_{1}=p_{2}=2$, we can find a bilinear form
$A\colon\ell_{p}^{n_{1}}\times\ell_{p}^{n_{2}}\longrightarrow\mathbb{C}$
satisfying
\begin{align*}
\Vert A\Vert &  \leq\left(  1+\epsilon\right)  ^{2\left(  1-\frac{1}%
{p}\right)  }\max\left\{  n_{1}^{1-\frac{1}{p}},n_{2}^{1-\frac{1}{p}}\right\}
\\
&  \leq\left(  1+\epsilon\right)  \max\left\{  n_{1}^{1-\frac{1}{p}}%
,n_{2}^{1-\frac{1}{p}}\right\} \\
&  =\left(  1+\epsilon\right)  \max\left\{  n_{2}^{1/p^{\ast}}n_{1}%
^{\max\left\{  \frac{1}{2}-\frac{1}{p},0\right\}  },n_{1}^{1/p^{\ast}}%
n_{2}^{\max\left\{  \frac{1}{2}-\frac{1}{p},0\right\}  }\right\}  \text{.}%
\end{align*}

\end{proof}

\begin{corollary}
\label{CorB6}Let $\epsilon>0$ and $1\leq p_{1}\leq p_{2}\leq2$ $\mathrm{(}%
1\leq p_{2}\leq p_{1}\leq2\mathrm{)}$. There exists a positive integer
$N$,$\mathrm{\ }$depending just on $\epsilon,$ such that, whenever $n_{2}>N$
$\mathrm{(}n_{1}>N\mathrm{)}$ and $n_{1}\leq n_{2}$ $\mathrm{(}n_{2}\leq
n_{1}\mathrm{)}$ there is a bilinear form $A\colon\ell_{p_{1}}^{n_{1}}%
\times\ell_{p_{2}}^{n_{2}}\longrightarrow\mathbb{K}$ of the type%
\[
A(z^{(1)},z^{(2)})=%
%TCIMACRO{\tsum \limits_{i=1}^{n_{1}}}%
%BeginExpansion
{\textstyle\sum\limits_{i=1}^{n_{1}}}
%EndExpansion%
%TCIMACRO{\tsum \limits_{j=1}^{n_{2}}}%
%BeginExpansion
{\textstyle\sum\limits_{j=1}^{n_{2}}}
%EndExpansion
\pm z_{i}^{(1)}z_{j}^{(2)}\text{,}%
\]
such that%
\[
\Vert A\Vert\leq\left(  1+\epsilon\right)  \max\left\{  n_{2}^{1/p_{2}^{\ast}%
}n_{1}^{\max\left\{  \frac{1}{2}-\frac{1}{p_{1}},0\right\}  },n_{1}%
^{1/p_{1}^{\ast}}n_{2}^{\max\left\{  \frac{1}{2}-\frac{1}{p_{2}},0\right\}
}\right\}  \text{.}%
\]

\end{corollary}

\begin{proof}
If $1\leq p_{1}\leq p_{2}\leq2$ and $n_{1}\leq n_{2}$, then%
\[
\max\left\{  n_{2}^{1/p_{2}^{\ast}}n_{1}^{\max\left\{  \frac{1}{2}-\frac
{1}{p_{1}},0\right\}  },n_{1}^{1/p_{1}^{\ast}}n_{2}^{\max\left\{  \frac{1}%
{2}-\frac{1}{p_{2}},0\right\}  }\right\}  =n_{2}^{1/p_{2}^{\ast}}\text{.}%
\]
By Corollary \ref{CorB5}, given $\epsilon>0$, there exists $N$ (depending just
on $\epsilon$) such that, if $n_{2}>N$, there is a bilinear form $A_{1}%
\colon\ell_{p_{2}}^{n_{1}}\times\ell_{p_{2}}^{n_{2}}\longrightarrow\mathbb{K}$
of the type%
\[
A_{1}(z^{\left(  1\right)  },z^{\left(  2\right)  })=%
%TCIMACRO{\tsum \limits_{i=1}^{n_{1}}}%
%BeginExpansion
{\textstyle\sum\limits_{i=1}^{n_{1}}}
%EndExpansion%
%TCIMACRO{\tsum \limits_{j=1}^{n_{2}}}%
%BeginExpansion
{\textstyle\sum\limits_{j=1}^{n_{2}}}
%EndExpansion
\pm z_{i}^{\left(  1\right)  }z_{j}^{\left(  2\right)  }\text{,}%
\]
satisfying%
\[
\Vert A_{1}\Vert\leq\left(  1+\epsilon\right)  n_{2}^{1/p_{2}^{\ast}}\text{.}%
\]
Now, define $A\colon\ell_{p_{1}}^{n_{1}}\times\ell_{p_{2}}^{n_{2}%
}\longrightarrow\mathbb{K}$ by $A\left(  x,y\right)  =A_{1}\left(  x,y\right)
$ and note that%
\[
\left\Vert A\right\Vert \leq\Vert A_{1}\Vert\leq\left(  1+\epsilon\right)
n_{2}^{1/p_{2}^{\ast}}\text{.}%
\]

\end{proof}

\begin{remark}
It is worth mentioning that there are still a few open cases concerning the
asymptotic behavior of the constants of Bennett's inequality. Our approach
does not cover bilinear forms $A\colon\ell_{p_{1}}^{n_{1}}\times\ell_{p_{2}%
}^{n_{2}}\longrightarrow\mathbb{K}$ with $1\leq p_{1}\leq p_{2}\leq2$
$\mathrm{(}1\leq p_{2}\leq p_{1}\leq2\mathrm{)}$ and $n_{1}>n_{2}$
$\mathrm{(}n_{2}>n_{1}\mathrm{)}$.
\end{remark}

%

%TCIMACRO{\TeXButton{Figura2}{\begin{figure}[H]
%\centering\begin{minipage}{6,9cm}
%\begin{tikzpicture}[x=2cm,y=2cm]
%\draw[yellow, fill=yellow,opacity=.2] (1,1)--(1,4)--(2,4)--(2,2)--(1,1);
%\draw[red, fill=red,opacity=.2] (1,1)--(4,1)--(4,2)--(2,2)--(1,1);
%\draw[green, fill=green,opacity=.2] (2,2)--(4,2)--(4,4)--(2,4)--(2,2);
%\draw[line width=1pt,color=green!80!black] (2,4)--(2,2)--(4,2);
%\draw[line width=1pt,color=red] (1,1)-- (2,2);
%\draw[->] (1,1)-- (1,4);
%\draw[->] (1,1)-- (4,1);
%\draw(1,1) node[below] {$1$};
%\draw(2,1) node[below] {$2$};
%\draw(4,1) node[below] {$p_{1}$};
%\draw(1,1) node[left] {$1$};
%\draw(1,2) node[left] {$2$};
%\draw(1,4) node[left] {$p_{2}$};
%\end{tikzpicture}
%\end{minipage}
%\begin{minipage}[bl]{5,8cm}
%\begin{enumerate}
%\item[{\color{green!80!black}$\blacksquare$}] Theorem \ref{TeoB3}
%for all $n_{1}, n_{2}$
%\item[{\color{green!20}$\blacksquare$}] Theorem \ref{TeoB3} for all $n_{1}%
%, n_{2}$
%\item[{\color{red!20}$\blacksquare$}] Corollaries \ref{CorB4} and \ref{CorB6}
%for $n_{2} \leq n_{1}$
%\item[{\color{yellow!20}$\blacksquare$}] Corollaries \ref{CorB4}
%and \ref{CorB6}
%for $n_{1} \leq n_{2}$
%\item[{\color{red}$\blacksquare$}] Corollary \ref{CorB5} for all $n_{1}
%, n_{2}$
%\end{enumerate}
%\end{minipage}
%\caption
%{The uniform asymptotic behavior of the constants of Bennett's inequality.}
%\label{Fig2}
%\end{figure}}}%
%BeginExpansion
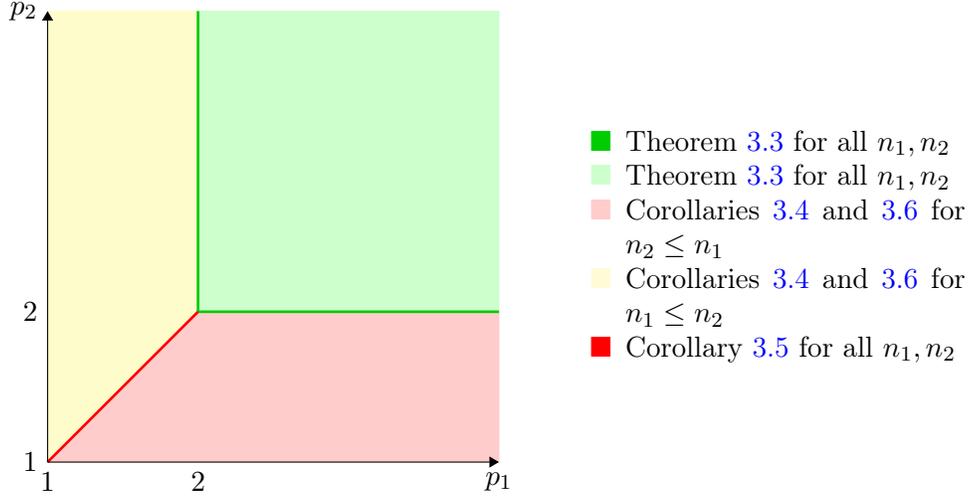
\begin{figure}[H]
\centering\begin{minipage}{6,9cm}
\begin{tikzpicture}[x=2cm,y=2cm]
\draw[yellow, fill=yellow,opacity=.2] (1,1)--(1,4)--(2,4)--(2,2)--(1,1);
\draw[red, fill=red,opacity=.2] (1,1)--(4,1)--(4,2)--(2,2)--(1,1);
\draw[green, fill=green,opacity=.2] (2,2)--(4,2)--(4,4)--(2,4)--(2,2);
\draw[line width=1pt,color=green!80!black] (2,4)--(2,2)--(4,2);
\draw[line width=1pt,color=red] (1,1)-- (2,2);
\draw[->] (1,1)-- (1,4);
\draw[->] (1,1)-- (4,1);
\draw(1,1) node[below] {$1$};
\draw(2,1) node[below] {$2$};
\draw(4,1) node[below] {$p_{1}$};
\draw(1,1) node[left] {$1$};
\draw(1,2) node[left] {$2$};
\draw(1,4) node[left] {$p_{2}$};
\end{tikzpicture}
\end{minipage}
\begin{minipage}[bl]{5,8cm}
\begin{enumerate}
\item[{\color{green!80!black}$\blacksquare$}] Theorem \ref{TeoB3}
for all $n_{1}, n_{2}$
\item[{\color{green!20}$\blacksquare$}] Theorem \ref{TeoB3} for all $n_{1}%
, n_{2}$
\item[{\color{red!20}$\blacksquare$}] Corollaries \ref{CorB4} and \ref{CorB6}
for $n_{2} \leq n_{1}$
\item[{\color{yellow!20}$\blacksquare$}] Corollaries \ref{CorB4}
and \ref{CorB6}
for $n_{1} \leq n_{2}$
\item[{\color{red}$\blacksquare$}] Corollary \ref{CorB5} for all $n_{1}
, n_{2}$
\end{enumerate}
\end{minipage}
\caption
{The uniform asymptotic behavior of the constants of Bennett's inequality.}
\label{Fig2}
\end{figure}%
%EndExpansion

As a final corollary, we show that the asymptotic constant $1$ is optimal when
$n_{1}=n_{2}$ and $p_{1},p_{2}\in\lbrack1,2]$:

\begin{corollary}
\label{op71}Let $p_{1},p_{2}\in\lbrack1,\infty]$ and $\epsilon>0.$ There is a
positive integer $N$ $\mathrm{(}$depending just on $\epsilon\mathrm{)}$ such
that, whenever $n>N$, there is a bilinear form $A\colon\ell_{p_{1}}^{n}%
\times\ell_{p_{2}}^{n}\longrightarrow\mathbb{K}$ of the type%
\begin{equation}
A(z^{(1)},z^{(2)})=%
%TCIMACRO{\tsum \limits_{i=1}^{n}}%
%BeginExpansion
{\textstyle\sum\limits_{i=1}^{n}}
%EndExpansion%
%TCIMACRO{\tsum \limits_{j=1}^{n}}%
%BeginExpansion
{\textstyle\sum\limits_{j=1}^{n}}
%EndExpansion
\pm z_{i}^{(1)}z_{j}^{(2)}\text{,} \label{mengao}%
\end{equation}
such that%
\[
\Vert A\Vert\leq\left(  1+\epsilon\right)  \max\left\{  n^{1/p_{2}^{\ast}%
}n^{\max\left\{  \frac{1}{2}-\frac{1}{p_{1}},0\right\}  },n^{1/p_{1}^{\ast}%
}n^{\max\left\{  \frac{1}{2}-\frac{1}{p_{2}},0\right\}  }\right\}
\]
and the constant $1$ is optimal whenever $p_{1},p_{2}\in\lbrack1,2]$.
\end{corollary}

\begin{proof}
All that remains to be proved is the optimality of the asymptotic constant
$1$. Let us show that $1$ cannot be replaced by a smaller constant when
$p_{1},p_{2}\in\lbrack1,2]$. Let $p=\max\{p_{1},p_{2}\}$. It is plain that for
all $A\colon\ell_{p_{1}}^{n}\times\ell_{p_{2}}^{n}\longrightarrow\mathbb{K}$
of the type (\ref{mengao}) we have
\[
n^{1/p^{\ast}}=\max_{i=1,...,n}\left(
%TCIMACRO{\tsum \limits_{j=1}^{n}}%
%BeginExpansion
{\textstyle\sum\limits_{j=1}^{n}}
%EndExpansion
\left\vert A\left(  e_{i},e_{j}\right)  \right\vert ^{p^{\ast}}\right)
^{\frac{1}{p^{\ast}}}\leq\left\Vert A\right\Vert .
\]
Therefore%
\[
\left\Vert A\right\Vert \geq n^{1/p^{\ast}}=\max\left\{  n^{1/p_{2}^{\ast}%
}n^{\max\left\{  \frac{1}{2}-\frac{1}{p_{1}},0\right\}  },n^{1/p_{1}^{\ast}%
}n^{\max\left\{  \frac{1}{2}-\frac{1}{p_{2}},0\right\}  }\right\}  \text{,}%
\]
and the proof is done.
\end{proof}

We conclude this section by remarking that, following the lines of the
statement of Theorem \ref{TeoB3}, we can re-state Theorem \ref{090} in a
slightly relaxed form as follows:

\begin{theorem}
Let a positive integer $m$ and $\epsilon>0$ be given. There is an $m$-linear
form $A_{n_{1},\ldots,n_{m}}\colon\ell_{\infty}^{n_{1}}\times\cdots\times
\ell_{\infty}^{n_{m}}\longrightarrow\mathbb{K}$ of the type
\[
A_{n_{1},\ldots,n_{m}}(z^{(1)},\ldots,z^{(m)})=%
%TCIMACRO{\tsum \limits_{j_{1}=1}^{n_{1}}}%
%BeginExpansion
{\textstyle\sum\limits_{j_{1}=1}^{n_{1}}}
%EndExpansion
\cdots%
%TCIMACRO{\tsum \limits_{j_{m}=1}^{n_{m}}}%
%BeginExpansion
{\textstyle\sum\limits_{j_{m}=1}^{n_{m}}}
%EndExpansion
\pm z_{j_{1}}^{(1)}\cdots z_{j_{m}}^{(m)}\text{,}%
\]
and a positive integer $N$ such that%
\[
\Vert A_{n_{1},\ldots,n_{m}}\Vert\leq\left(  1+\epsilon\right)  \max\left\{
n_{1}^{1/2},\ldots,n_{m}^{1/2}\right\}
%TCIMACRO{\tprod \limits_{j=1}^{m}}%
%BeginExpansion
{\textstyle\prod\limits_{j=1}^{m}}
%EndExpansion
n_{j}^{1/2}\text{,}%
\]
whenever
\[
\max_{i=1,\ldots,m}\left\{  \min\left\{  n_{j}:j\neq i\right\}  \right\}
>N\text{.}%
\]

\end{theorem}

\section{Remarks}

We begin this section by remarking that well known probabilistic techniques
used in the proof of the classical Kahane--Salem--Zygmund inequalities,
combined with Lemma \ref{LemaB1}, improve the estimates obtained by
Albuquerque and Rezende in \cite{alb} for the remaining cases not covered in
the previous sections. In fact, combining the proof of Mantero and Tonge
\cite[Theorem 1.1]{man} for $p_{1}=\cdots=p_{m}=2$ with Lemma \ref{LemaB1} we
can easily improve the constants of (\ref{I4}) and (\ref{I5}). More precisely,
following the proof of \cite[Theorem 1.1]{man} (noting that the constant $K$
in the proof of \cite[Theorem 1.1]{man} can be chosen as the degree of
multilinearity (now denoted by $m$) and combining with Lemma \ref{LemaB1}, we
can replace the constants of (\ref{I4}) and (\ref{I5}) by
\[
C_{m}\leq2^{m+1}\sqrt{\left(  2m+1\right)  \log\left(  1+4m\right)  }\text{.}%
\]
In the case of real scalars we can avoid the factor $2^{m}$ and get%
\[
C_{m}\leq2\sqrt{\left(  2m+1\right)  \log\left(  1+4m\right)  }\text{.}%
\]

However, it is still an open problem if it is possible to obtain constants
asymptotically dominated by $1$ in these cases.

We also observe that a simple adaptation of the arguments from Lemma
\ref{LemaB1} shows that the constants of the family of inequalities (\ref{I5})
are monotone in the following sense: if $C_{m,p}$ is the optimal constant of
(\ref{I5}) for $p_{1}=\cdots=p_{m}=p$, then%
\[
q\geq p\geq2\Rightarrow C_{m,q}\leq C_{m,p}\text{.}%
\]
We finish this section by remarking that the apparent gap between the results
from Theorems \ref{090} and \ref{TeoB3} is somewhat justified by the following
theorem which follows the exact same lines of the proof of Theorem \ref{090}:

\begin{theorem}
Let $m$ be a positive integer, $p_{1},p_{m}\in\left[  2,\infty\right]  $,
$p_{2}=\cdots=p_{m-1}=\infty$ and $\epsilon>0$. There exists a positive
integer $N$ such that, for any choice of positive integers $n_{1},\ldots
,n_{m}>N$, with $n_{m}=\max\left\{  n_{1},\ldots,n_{m}\right\}  $ and
$n_{1}=\min\left\{  n_{1},\ldots,n_{m}\right\}  $, there exists an $m$-linear
form $A\colon\ell_{p_{1}}^{n_{1}}\times\ell_{p_{2}}^{n_{2}}\times\cdots
\times\ell_{p_{m-1}}^{n_{m-1}}\times\ell_{p_{m}}^{n_{m}}\longrightarrow
\mathbb{K}$ of the type
\[
A(z^{(1)},\ldots,z^{(m)})=%
%TCIMACRO{\tsum \limits_{i_{1}=1}^{n_{1}}}%
%BeginExpansion
{\textstyle\sum\limits_{i_{1}=1}^{n_{1}}}
%EndExpansion
\cdots%
%TCIMACRO{\tsum \limits_{i_{m}=1}^{n_{m}}}%
%BeginExpansion
{\textstyle\sum\limits_{i_{m}=1}^{n_{m}}}
%EndExpansion
\pm z_{i_{1}}^{(1)}\cdots z_{i_{m}}^{(m)}\text{,}%
\]
satisfying%
\[
\Vert A\Vert\leq\left(  1+\epsilon\right)  n_{m}^{1/2}%
%TCIMACRO{\tprod \limits_{k=1}^{m}}%
%BeginExpansion
{\textstyle\prod\limits_{k=1}^{m}}
%EndExpansion
n_{k}^{\frac{1}{2}-\frac{1}{p_{k}}}\text{.}%
\]

\end{theorem}

\section{Application: Berlekamp's switching game}

Berlekamp's switching game (sometimes called Gale-Berlekamp switching game or
unbalancing lights problem, see \cite[Section 2.5]{alon} and \cite[Chapter
6]{spencer}) consists of an $n\times n$ square matrix of light bulbs set up at
an initial configuration $\Theta_{n}$. The board has $n$ row and $n$ column
switches, which invert the on-off state of each bulb (on to off and off to on)
in the corresponding row or column. Let $i(\Theta_{n})$ denote the smallest
final number of on-lights achievable by row and column switches starting from
$\Theta_{n}$. The goal is to find the value $R_{n}$ of $i(\Theta_{n}^{(0)})$
when $\Theta_{n}^{(0)}$ is (one of) the worst initial patterns, i.e.,
$i(\Theta_{n}^{(0)})\geq i(\Theta_{n})$ for all $\Theta_{n}$. Thus
\[
R_{n}:=\max\{i(\Theta_{n}):\Theta_{n}\text{ is an initial configuration of
}n\times n\text{ lights}\}\text{.}%
\]
Sometimes the problem is posed as to find the maximum of the difference
between the state of the light bulbs (starting from one of the worst initial
patterns, as before), which we shall henceforth denote by $G_{n}$. It is
simple to check that
\[
R_{n}=\frac{1}{2}\left(  n^{2}-G_{n}\right)  \text{.}%
\]
The determination of the exact value of $R_{n}$ seems to be conceivable only
for small values of $n$ due to involving combinatorial arguments. The exact
value of $R_{n}$ for $n$ up to $12$ was obtained by Carlson and Stolarski
(\cite{car}; see also \cite{TBJS,slo}). For bigger values of $n$, optimal
constructive approaches seem impracticable and no algorithm to construct such
a \textquotedblleft bad\textquotedblright\ configuration $\Theta_{n}$ seems to
be known. Thus, for bigger values of $n$, probabilistic (non-deterministic)
methods are used to provide estimates for $R_{n}$ and $G_{n}$. The natural
approach to modeling Berlekamp's switching game is by associating $+1$ to the
on-lights and $-1$ to the off-lights from the array of lights $\left(
a_{ij}\right)  _{i,j=1}^{n}$ and observing that
\[
G_{n}=\min\left\{  \max_{\left(  x_{i}\right)  _{i=1}^{n},\left(
y_{j}\right)  _{j=1}^{n}\in\{-1,1\}^{n}}\left\vert
%TCIMACRO{\tsum \limits_{i,j=1}^{n}}%
%BeginExpansion
{\textstyle\sum\limits_{i,j=1}^{n}}
%EndExpansion
a_{ij}x_{i}y_{j}\right\vert :a_{ij}=-1\text{ or }+1\right\}  \text{,}%
\]
where $x_{i}$ and $y_{j}$ denote the switches of row $i$ and of column $j$, respectively.

Berlekamp's switching game has several natural variants (see, for instance,
\cite{BS,BRUA,PST} and the references therein); we also refer to \cite{kalai}
for a recent related result. We shall be interested in its extension to higher
dimensions (see \cite{ara}). Let $m\geq2$ be an integer and let an
$n\times\cdots\times n$ array $\left(  a_{j_{1}\cdots j_{m}}\right)  $ of
lights be given, each either on ($a_{j_{1}\cdots j_{m}}=1$) or off
($a_{j_{1}\cdots j_{m}}=-1$). Let us also suppose that for each $k=1,\ldots,m$
and each $j_{k}=1,\ldots,n$ there is a switch $x_{j_{k}}^{(k)}$ so that if the
switch is pulled ($x_{j_{k}}^{(k)}=-1$) all of the corresponding lights
$a_{j_{1}\ldots j_{m}}$ (with $j_{k}$ fixed) are switched: on to off or off to
on. The goal is to maximize the difference between the number of lights that
are on and the number of lights that are off.

\begin{figure}[ptb]
\centering\includegraphics[scale=0.56]{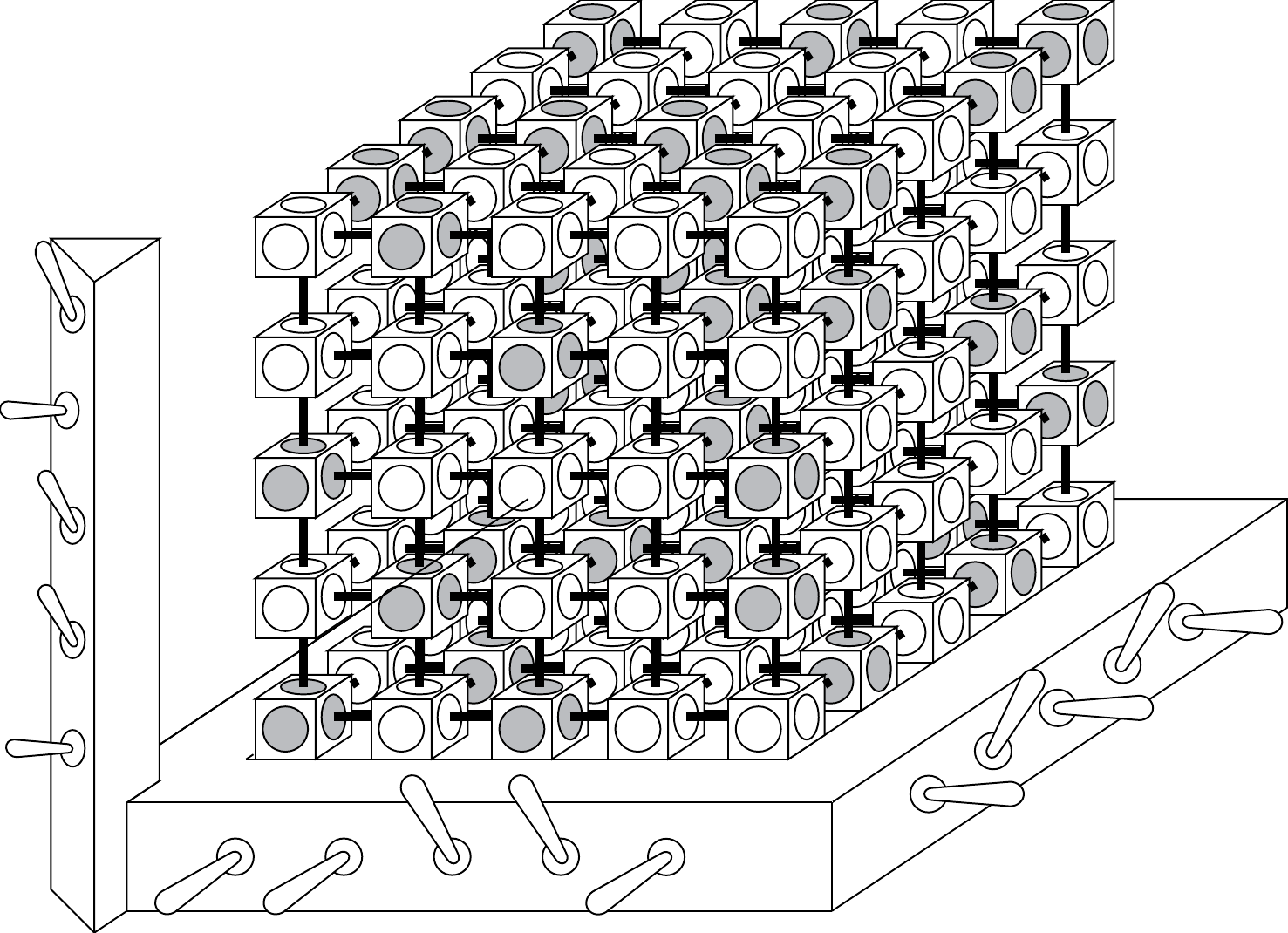}\caption{Three-dimensional
Berlekamp's switching game for $n=5$.}%
\end{figure}
%EndExpansion

As in the two-dimensional case, maximizing the difference between the number
of on-lights and off-lights is equivalent to estimating
\[
\max_{(x_{j_{1}}^{(1)})_{j_{1}=1}^{n},\ldots,(x_{j_{m}}^{(m)})_{j_{m}=1}%
^{n}\in\{-1,1\}^{n}}\left\vert
%TCIMACRO{\tsum \limits_{j_{1},\ldots,j_{m}=1}^{n}}%
%BeginExpansion
{\textstyle\sum\limits_{j_{1},\ldots,j_{m}=1}^{n}}
%EndExpansion
a_{j_{1}\ldots j_{m}}x_{j_{1}}^{(1)}\cdots x_{j_{m}}^{(m)}\right\vert
\]
and the extreme problem consists of estimating
\[
S_{m,n}=\min\left\{  \max_{(x_{j_{1}}^{(1)})_{j_{1}=1}^{n},\ldots,(x_{j_{m}%
}^{(m)})_{j_{m}=1}^{n}\in\{-1,1\}^{n}}\left\vert
%TCIMACRO{\tsum \limits_{j_{1},\ldots,j_{m}=1}^{n}}%
%BeginExpansion
{\textstyle\sum\limits_{j_{1},\ldots,j_{m}=1}^{n}}
%EndExpansion
a_{j_{1}\ldots j_{m}}x_{j_{1}}^{\left(  1\right)  }\cdots x_{j_{m}}^{\left(
m\right)  }\right\vert :a_{j_{1}\ldots j_{m}}=1\text{ or }-1\right\}  \text{,}%
\]
It is folklore (a consequence of the Krein--Milman Theorem) that
\begin{equation}
S_{m,n}=\min\left\Vert A_{m,n}\colon\ell_{\infty}^{n}\times\cdots\times
\ell_{\infty}^{n}\longrightarrow\mathbb{R}\right\Vert \text{,} \label{uzz}%
\end{equation}
with%
\[
A_{m,n}\left(  x^{(1)},\ldots,x^{(m)}\right)  =%
%TCIMACRO{\tsum \limits_{j_{1},\ldots,j_{m}=1}^{n}}%
%BeginExpansion
{\textstyle\sum\limits_{j_{1},\ldots,j_{m}=1}^{n}}
%EndExpansion
a_{j_{1}\ldots j_{m}}x_{j_{1}}^{(1)}\cdots x_{j_{m}}^{(m)}\text{.}%
\]

As an illustrative example of how the estimates of the constants of the KSZ
inequality are associated to this framework, note that, combining (\ref{uzz})
and Lemma \ref{M1}, we have $G_{16}\leq64$ and hence $R_{16}\geq96,$ improving
the estimate $R_{16}\geq94$ given by \cite[Theorem 8]{carl}.

Let us choose, as usual in this setting, the notation $o(1)$ to represent
that, in the variable $n$, we have $\lim\limits_{n\rightarrow\infty}o(1)=0$.
The results of Section 2 show that
\[
S_{m,n}\leq\left(  1+o(1)\right)  n^{\frac{m+1}{2}}\text{,}%
\]
and this result improves the best known constants to this problem, which were
obtained by the KSZ inequality, i.e.,
\[
S_{m,n}\leq\left(  \sqrt{32m\log\left(  6m\right)  }\sqrt{m!}\right)
n^{\frac{m+1}{2}}\text{.}%
\]

Combining our estimates with those in \cite[Theorem 3.2]{ara} and
\cite[Theorem 2.5.1]{alon}, we obtain
\[
2^{-\frac{1}{2}\psi\left(  m\right)  -\frac{1}{2}\gamma}%
%TCIMACRO{\tprod \limits_{k=2}^{m}}%
%BeginExpansion
{\textstyle\prod\limits_{k=2}^{m}}
%EndExpansion
\left(  \dfrac{\Gamma\left(  \frac{3k-2}{2k}\right)  }{\Gamma\left(  \frac
{3}{2}\right)  }\right)  ^{\frac{k}{2k-2}}+o\left(  1\right)  \leq
\dfrac{S_{m,n}}{n^{\frac{m+1}{2}}}\leq1+o\left(  1\right)  \text{,}%
\]
where $\psi$ is the digamma function and $\gamma$ is the Euler-Mascheroni
constant. For instance, for $m=2,3$ we have%
\begin{align*}
0.797+o\left(  1\right)   &  \leq\dfrac{S_{2,n}}{n^{3/2}}\leq1+o\left(
1\right)  \text{,}\\
0.694+o\left(  1\right)   &  \leq\dfrac{S_{3,n}}{n^{2}}\leq1+o\left(
1\right)  \text{.}%
\end{align*}

\bigskip

\medskip\noindent\textbf{Acknowledgment.} The authors thank Fernando Vieira
Costa Jr and Janiely Silva for their support and generosity.

\end{document}